\documentclass[fleqn,reqno,11pt,a4paper,final]{amsart}

\usepackage[a4paper,left=20mm,right=20mm,top=20mm,bottom=20mm,marginpar=20mm]{geometry}
\usepackage[english]{babel}
\usepackage{amsmath}
\usepackage{amssymb}
\usepackage{amsthm}
\usepackage{amscd}
\usepackage[utf8]{inputenc}
\usepackage{color}
\usepackage[english=american]{csquotes}
\usepackage[final]{graphicx}
\usepackage{hyperref}
\usepackage{calc}
\usepackage{mathptmx}
\usepackage{mathtools}
\usepackage{textcmds} % for quotation marks
\usepackage{tikz}
\usepackage{graphicx}
\usepackage{float}
\usepackage{esint}
\usepackage{enumitem}
\usepackage{dsfont}

\usepackage[backend=biber,bibencoding=utf8,firstinits,maxnames=4,style=alphabetic,isbn=false, doi=false, eprint= true, url= false,date=year]{biblatex}
\usepackage[english=american]{csquotes}
\bibliography{literature.bib}

\makeatletter
 \newcommand{\linkdest}[2]{\Hy@raisedlink{\hypertarget{#1}{#2}}}
\makeatother

\definecolor{luh-dark-blue}{rgb}{0.0, 0.313, 0.608}

\graphicspath{{../Pictures/}}

\numberwithin{equation}{section}

\newtheoremstyle{thmlemcorr}{10pt}{10pt}{\itshape}{}{\bfseries}{.}{10pt}{{\thmname{#1}\thmnumber{ #2}\thmnote{ (#3)}}}
\newtheoremstyle{thmlemcorr*}{10pt}{10pt}{\itshape}{}{\bfseries}{.}\newline{{\thmname{#1}\thmnumber{ #2}\thmnote{ (#3)}}}
\newtheoremstyle{remexample}{10pt}{10pt}{}{}{\bfseries}{.}{10pt}{{\thmname{#1}\thmnumber{ #2}\thmnote{ (#3)}}}
\newtheoremstyle{ass}{10pt}{10pt}{}{}{\bfseries}{.}{10pt}{{\thmname{#1}\thmnumber{ A#2}\thmnote{ (#3)}}}

\theoremstyle{thmlemcorr}
\newtheorem{theorem}{Theorem}
\numberwithin{theorem}{section}
\newtheorem{lemma}[theorem]{Lemma}
\newtheorem{corollary}[theorem]{Corollary}
\newtheorem{proposition}[theorem]{Proposition}

\newtheorem{definition}[theorem]{Definition}

\theoremstyle{thmlemcorr*}
\newtheorem*{theorem*}{Theorem}
\newtheorem{lemma*}[theorem]{Lemma}
\newtheorem{corollary*}[theorem]{Corollary}
\newtheorem{proposition*}[theorem]{Proposition}
\newtheorem{problem*}[theorem]{Problem}
\newtheorem{conjecture*}[theorem]{Conjecture}
\newtheorem{definition*}[theorem]{Definition}
\newtheorem{assumption*}[theorem]{Assumption}

\theoremstyle{remexample}
\newtheorem{remark}[theorem]{Remark}

\theoremstyle{ass}

\newcommand{\Fcal}{\mathcal{F}}
\newcommand{\Gcal}{\mathcal{G}}
\newcommand{\Hcal}{\mathcal{H}}

\newcommand{\Kcal}{\mathcal{K}}
\newcommand{\Lcal}{\mathcal{L}}

\newcommand{\Ncal}{\mathcal{N}}
\newcommand{\Ocal}{\mathcal{O}}

\newcommand{\Rcal}{\mathcal{R}}
\newcommand{\Scal}{\mathcal{S}}

\newcommand{\Ucal}{\mathcal{U}}

\newcommand{\Xcal}{\mathcal{X}}
\newcommand{\Ycal}{\mathcal{Y}}

\newcommand{\de}{\, \mathrm{d}}

\renewcommand{\div}{\operatorname{div}}

\DeclareMathOperator{\spn}{span}
\DeclareMathOperator{\ran}{ran}

\renewcommand{\Re}{\operatorname{Re}}
\renewcommand{\Im}{\operatorname{Im}}

\newcommand{\N}{\mathbb{N}}
\newcommand{\R}{\mathbb{R}}
\newcommand{\C}{\mathbb{C}}

\newcommand{\Z}{\mathbb{Z}}

\newcommand{\loc}{\mathrm{loc}}

\newcommand{\per}{\mathrm{per}}

\newcommand{\eps}{\varepsilon}

\DeclareMathOperator{\Id}{Id}

\def\div{\mathrm{div\,}}

\def\XXint#1#2#3{{\setbox0=\hbox{$#1{#2#3}{\int}$}
\vcenter{\hbox{$#2#3$}}\kern-.5\wd0}}

\renewcommand{\eps}{\varepsilon}
\renewcommand{\phi}{\varphi}
\begin{document}

%% TITLE MATTERS

\title[]{Thermocapillary Thin Films: Periodic Steady States and Film Rupture}

\author{Gabriele Bruell}
\address{\textit{Gabriele Bruell:} Centre for Mathematical Sciences, Lund University, P.O. Box 118, 221 00 Lund, Sweden}
\email{gabriele.brull@math.lth.se}

\author{Bastian Hilder}
\address{\textit{Bastian Hilder:}  Department of Mathematics, Technische Universität München, Boltzmannstraße 3, 85748 Garching b.~München, Germany}
\email{bastian.hilder@tum.de}

\author{Jonas Jansen}
\address{\textit{Jonas Jansen:}  Centre for Mathematical Sciences, Lund University, P.O. Box 118, 221 00 Lund, Sweden}
\email{jonas.jansen@math.lth.se}

\begin{abstract}
    We study stationary, periodic solutions to the thermocapillary thin-film model
    \begin{equation*}
    \partial_t h + \partial_x \Bigl(h^3(\partial_x^3 h - g\partial_x h) + M\frac{h^2}{(1+h)^2}\partial_xh\Bigr) = 0,\quad t>0,\ x\in \R,
    \end{equation*}
    which can be derived from the B\'enard--Marangoni problem via a lubrication approximation. When the Marangoni number \(M\) increases beyond a critical value \(M^*\), the constant solution becomes spectrally unstable via a (conserved) long-wave instability and periodic stationary solutions bifurcate. For a fixed period, we find that these solutions lie on a global bifurcation curve of stationary, periodic solutions with a fixed wave number and mass. Furthermore, we show that the stationary periodic solutions on the global bifurcation branch converge to a weak stationary periodic solution which exhibits film rupture. The proofs rely on a Hamiltonian formulation of the stationary problem and the use of analytic global bifurcation theory. Finally, we show %study 
    the instability of the bifurcating solutions close to the bifurcation point and give a formal derivation of the amplitude equation governing the dynamics close to the onset of instability. 
\end{abstract}
\vspace{4pt}

\maketitle



%\vspace{4pt}

%\noindent\textsc{Date:} \today{}.
%\end{abstract}

%% PDF MATTERS

%% START OF CONTENT

%=============================================================================
%=============================================================================
%=============================================================================

\section{Introduction}

Instabilities in thin fluid films triggered by temperature gradients can lead to the formation of interesting nonlinear waves including pattern formation. Since the first experiments of B\'enard \cite{bénard1900,benard1901} in 1900 and 1901, understanding these instabilities has attracted a lot of attention. Pearson \cite{pearson1958} found that the instabilities and the resulting hexagonal patterns observed by B\'enard are due to a temperature dependence in the surface tension. Accordingly, these instabilities are called \emph{thermocapillary instabilities}.

The present manuscript is concerned with a thin fluid film located on a heated plane with a free top surface. In this setting two main effects can be observed. The first one is the emergence of periodic, hexagonal convection cells \cite{schatz1995}. Since these patterns have minimal surface deformations, this instability is called non-deformational instability. The second one is that spontaneous film rupture can occur, see \cite{vanhook1995,vanhook1997} for experimental results and \cite{krishnamoorthy1995} for a numerical study. This is referred to as a deformational instability.

As a first step to understand the deformational instability, we study the one-dimensional thermocapillary thin-film model
\begin{equation}\label{eq:thin-film-equation}
    \partial_t h + \partial_x \Bigl(h^3(\partial_x^3 h - g\partial_x h) + M\frac{h^2}{(1+h)^2}\partial_xh\Bigr) = 0,\quad t>0,\ x\in \R.
\end{equation}
Here, \(h=h(t,x)\) denotes the height of a thin fluid film on an impermeable heated flat plane. Additionally, \(g\) is a gravitational constant and \(M\) a scaled Marangoni number. Using a lubrication approximation, equation \eqref{eq:thin-film-equation} can formally be derived from the full B\'enard--Marangoni problem, that is the Navier--Stokes equations coupled with a transport-diffusion equation for the temperature, see Section \ref{sec:modelling} for details.

Beyond the physical importance, equation \eqref{eq:thin-film-equation} is also interesting from a pure mathematical perspective since it is a quasilinear fourth-order partial differential equation which is degenerate-parabolic in the film height.

The aim of this paper is a rigorous mathematical analysis of stationary solutions to \eqref{eq:thin-film-equation}. In particular, we study the existence and stability of stationary periodic solutions to \eqref{eq:thin-film-equation}. For Marangoni numbers beyond a critical value, we prove the global bifurcation of periodic stationary solutions, emerging from the spectrally unstable constant surface. Asympotically, the solutions on the global bifurcation branch converge to a perodic stationary solution to \eqref{eq:thin-film-equation} exhibiting film rupture. While this establishes the existence of static film-rupture solutions, the dynamical formation of film rupture remains an open problem.

\begin{figure}[h]
    \centering
    \includegraphics[width = 0.9\textwidth]{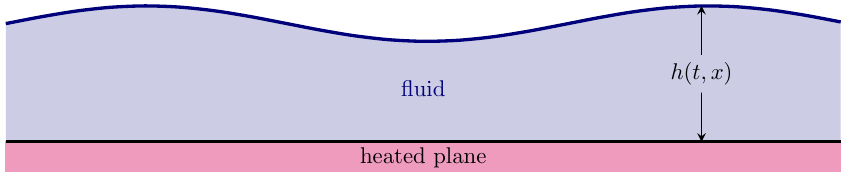}
    \caption{Thin fluid film resting on a heated plane.}
    \label{fig:thin-film}
\end{figure}

\subsection{Main results of the paper}

We summarise the main results and techniques of this paper. We prove that

\begin{itemize}
    \item there is a critical Marangoni number \(M^*=4g\) at which the steady state $\bar{h} = 1$ destabilizes and  for every \(k_0>0\) there exists a global bifurcation branch, bifurcating at \(M=M^* + 4k_0^2\), of $\tfrac{2\pi}{k_0}$-periodic stationary solutions to \eqref{eq:thin-film-equation} which are bounded in \(W^{2,p}_{\mathrm{per}}\) for all \(p\in [1,\infty)\) and the minimum approaches \(0\) along the bifurcation curve;
    \item for every $k_0 > 0$ there exist $M_\infty > 0$ and $h_\infty \in W^{2,p}_\mathrm{per}$ for all \(1\leq p <\infty\) which is a periodic stationary weak solution to \eqref{eq:thin-film-equation}  with wave number $k_0$ and satisfies $h_\infty\bigl(\pm \tfrac{\pi}{k_0}\bigr) = 0$;
    \item at the critical Marangoni number \(M^*\) the constant steady-state solution destabilises via a (conserved) long-wave instability, see Figure \ref{fig:spectral-curves}, and is spectrally unstable for \(M>M^*\). Furthermore, the bifurcating periodic solution is spectrally unstable close to the bifurcation point.
\end{itemize}

\begin{remark}
    Throughout the paper and without loss of generality, we study only solutions, which bifurcate from the constant steady state $\bar{h} = 1$.
\end{remark}

\noindent\textbf{\textsc{Global bifurcation of periodic steady states. }} Since the constant steady-state solution destabilises at $M^*$ via a (conserved) long-wave instability, we expect that periodic stationary solutions to \eqref{eq:thin-film-equation} with arbitrary period bifurcate from the constant state for $M > M^\ast$. We fix a wave number $k_0 > 0$ and study the existence of $\tfrac{2\pi}{k_0}$-periodic solutions using classical bifurcation theory, see e.g.~\cite{buffoni2003,kielhöfer2012}. The core point in our analysis is that, after two integrations, the stationary problem for \eqref{eq:thin-film-equation} can be written as a second-order differential equation
\begin{equation}\label{eq:integrated-eq-introduction}
    \partial_x^2 v = gv - M\Bigl(\frac{1}{2+v} + \log\Bigl(\frac{1+v}{2+v}\Bigr) \Bigr) + MK,
\end{equation}
with $v = h - 1$ and a constant of integration $K$. The constant $K$ changes along the bifurcation curve, which is used to guarantee that the mass over one period of $v$ is conserved. 

Most importantly for the subsequent analysis, it turns out that \eqref{eq:integrated-eq-introduction} has a Hamiltonian structure with Hamiltonian
\begin{equation*}
    \Hcal(v,w) = \frac{1}{2}w^2 - \frac{g}{2} v^2 + M(1+v)\log\Bigl(\frac{1+v}{2+v}\Bigr) - MKv,
\end{equation*}
where $w = v'$. Using this structure, the dynamics of \eqref{eq:integrated-eq-introduction} can be understood using phase-plane analysis by exploiting the fact that all orbits lie on the level sets of the Hamiltonian, see Lemma \ref{lem:orbits-Hamiltonian-system} and Figure \ref{fig:phase-space}. It turns out that the Hamiltonian system has two fixed points, a center and a saddle point. The main observation used in out analysis is that the center has a neighborhood of periodic orbits, which is enclosed by the level set of the saddle point, see Figure \ref{fig:phase-space}.

In particular, linearising the Hamiltonian system about $v = 0$, we obtain the eigenvalues $\lambda_\pm = \pm \sqrt{g-\tfrac{1}{4}M}$. Hence, at $M = M^*$ the two eigenvalues collide at $\lambda = 0$ and split up into two purely imaginary eigenvalues for $M > M^*$. Thus, due to the Lyapunov subcenter theorem, see e.g.~\cite[Thm. 4.1.8]{schneider2017}, we expect the bifurcation of $\tfrac{2\pi}{k_0}$-periodic solutions at $M^*(k_0) = M^* + 4k_0^2$.

We make this intuition rigorous by studying the bifurcation problem for $\tfrac{2\pi}{k_0}$-periodic, even solutions to \eqref{eq:integrated-eq-introduction}. It turns out that the restriction to even solutions removes the translation invariance and results in a one-dimensional kernel and the bifurcation of periodic solutions can be established by the Crandall--Rabinowitz theorem, see Theorem \ref{thm:local-bifurcation} for details.

We then extend the local curve to a global bifurcation branch using analytic global bifurcation theory, see Theorem \ref{thm:global-bifurcation} and Proposition \ref{prop:global-bifurcation}. Typically, the main challenge is then to study the behaviour of solutions along the bifurcation curve since the general global bifurcation theorem only establishes that at least one of the following alternatives must hold. The curve forms a closed loop, the solutions blow-up along the curve or the curve approaches the boundary of the phase space, see Proposition \ref{prop:global-bifurcation}.

The scenario that the global bifurcation curve forms a closed loop is ruled out by nodal properties (Proposition \ref{prop:no-closed-loop}). This refers to the observation that solutions on the global bifurcation curve are symmetric and increasing on $(-\tfrac{\pi}{k_0},0)$. Therefore, the minimal and maximal points of the solutions are invariant along the global bifurcation curve and the solutions form a nodal pattern, see Remark~\ref{rem:nodal-property}. This also establishes that the periodic solutions have a global minimum at $x = \pm \tfrac{\pi}{k_0}$. Furthermore, using phase-plane analysis of the Hamiltonian system and elliptic regularity theory, we prove that the solutions along the bifurcation curve stay bounded in $H^2_{\mathrm{per}}\times \R$, see Lemma \ref{lem:uniform-bound-M} and Proposition \ref{prop:uniform-bound-H2}. Therefore, we may conclude that the bifurcation curve approaches the boundary of the phase plane. In the present setting, this corresponds to the case that $v$ approaches $-1$. In view of $h = v+1$, it means that the global bifurcation curve of stationary, periodic solutions to \eqref{eq:thin-film-equation} approaches a film-rupture state, see Proposition \ref{prop:film-rupture}.

The qualitative structure of these results is presented in Figure \ref{fig:bifurcation-diagram} and the main result is summarised in Theorem \ref{thm:main-global-bifurcation} below.

\begin{figure}[H]
    \centering
        \includegraphics[width=0.8\textwidth]{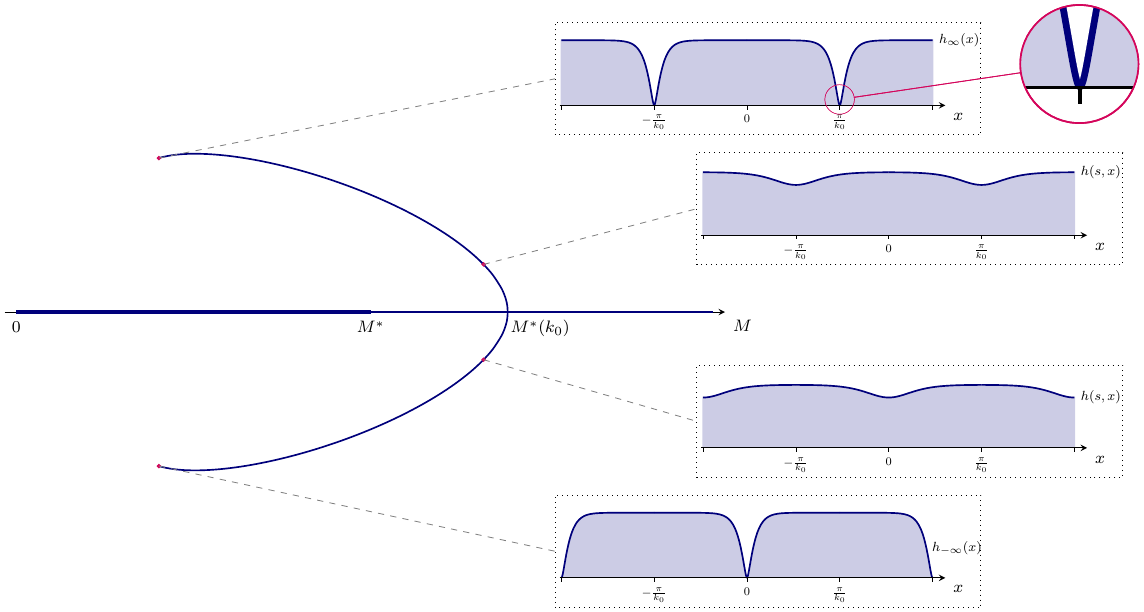}
    \caption{Schematic depiction of the bifurcation diagram of periodic, stationary solutions obtained in this paper. The constant solution is spectrally stable for $M<M^*$ and spectrally unstable for $M>M^*$. At $M^*(k_0)$, periodic solutions with wave number $k_0$ bifurcate subcritically from the constant state. Along the bifurcation curve the minimum of the periodic solution approaches zero.} 
    \label{fig:bifurcation-diagram}
\end{figure}

\begin{theorem}[Global bifurcation]\label{thm:main-global-bifurcation}
    Let \(k_0 > 0\) be a given wave number and \(M^*(k_0) = M^* + 4k_0^2\). 
    Then at \((v,M)=(0,M^\ast(k_0))\) a subcritical pitchfork bifurcation occurs and the corresponding bifurcation branch can be extended to a global bifurcation curve
    \begin{equation*}
        \{(v(s),M(s)) : s \in \R\} \subset \Ucal \times \R := \Bigl\{v\in H^2_{\mathrm{per}} \,:\, v\text{ even, }\int_{-\frac{\pi}{k_0}}^{\frac{\pi}{k_0}} v\de x = 0 \text{ and } v > -1\Bigr\} \times \R
    \end{equation*}
    of $\tfrac{2\pi}{k_0}$-periodic stationary solutions to \eqref{eq:thin-film-equation}.
    The bifurcation curve is uniformly bounded in $W^{2,p}_\mathrm{per}$ for all \(p\in [1,\infty)\) and does not return to the trivial curve.
    Moreover, film rupture occurs, that is $v(s)$ satisfies
    \begin{equation*}
        \inf_{s\in \R} \min_{x\in \bigl[-\tfrac{\pi}{k_0},\tfrac{\pi}{k_0}\bigr]}v(s) = -1.
    \end{equation*}
\end{theorem}

\begin{remark}
    The bifurcation curve in the schematic depiction in Figure \ref{fig:bifurcation-diagram} extends below $M^\ast$. This can indeed be verified numerically via continuation using pde2path, see \cite{uecker2014,uecker2021} and \cite{zotero-3112}. The numerical results are depicted in Figure \ref{fig:numerical-continuation-plots}. In particular, it suggests that stationary periodic solutions close to film rupture occur at $M < M^\ast$ for which the constant steady state is stable. However, an analytic proof of this observation is still open.
\end{remark}

\noindent\textbf{\textsc{Film rupture. }} Due to the uniform bounds obtained in the analysis of the global bifurcation branch, we obtain limit points $h_{\infty}\in W^{2,p}_{\mathrm{per}}$ for all $p \in [1,\infty)$ and $M_{\infty}>0$ of the branch. Furthermore, we know that $h_{\infty}$ exhibits film rupture and thus, if it is a periodic stationary solution to \eqref{eq:thin-film-equation}, its second derivative blows up at the film-rupture points. In particular, $h_\infty$ is not in $W^{2,\infty}_\mathrm{per}$. Hence, $h_\infty$ cannot be a classical solution to \eqref{eq:thin-film-equation} and in order to show that these limit points are still stationary solutions to \eqref{eq:thin-film-equation}, we introduce the following weak formulation.

\begin{definition}\label{def:weak-solution}
    A function \(h\in H^1_{\loc}(\R)\) with \(\partial_x^3 h \in L^2_{\loc}(\{h>0\})\) and \(h^3(\partial_x^3h-g\partial_x h) + M \frac{h^2}{(1+h)^2}\partial_x h \in L^2_{\loc}(\R)\) is a weak stationary solution to the thin-film equation \eqref{eq:thin-film-equation} if
\begin{equation*}
    \int_{\{h>0\}} \bigl(h^3(\partial_x^3h-g\partial_x h) + M \frac{h^2}{(1+h)^2}\partial_x h\bigr)\partial_x \phi \de x = 0
\end{equation*}
holds for all \(\phi\in H^1(\R)\) with compact support.
\end{definition}

In order to pass to the limit in the nonlinear flux $h^3(\partial_x^3 h - g\partial_x h) + M\tfrac{h^2}{(1+h)^2}\partial_xh$ of \eqref{eq:thin-film-equation}, we obtain uniform bounds on \(\partial_x^3 h(s)\) locally on the positivity set of \(h_{\infty}\). Combining this with the uniform bound on the nonlinear term provided by the equation, we show that \(h_{\infty}\) is a periodic weak stationary solution to \eqref{eq:thin-film-equation}. Moreover, $h_{\infty}^3\partial_x^3 h_{\infty}$ can be extended continuously into the set \(\{h_{\infty}=0\}\) which consists precisely of \(x=\pm \frac{\pi}{k_0} + 2j\tfrac{\pi}{k_0}\), \(j\in \Z\). We summarise this result in the following theorem (see also Theorem \ref{thm:weak-solutions} for more details).

\begin{theorem}[Film rupture]\label{thm:main-film-rupture}
    For all wave numbers $k_0 > 0$ exist \(M_{\infty} > 0\) and a function \(h_{\infty} \in W^{2,p}_{\mathrm{per}}\) for all \(p\in [1,\infty)\) with \(\partial_x^3 h_{\infty} \in L^2_{\loc}(\{h_{\infty} >0\})\) such that  \(h_{\infty}\) is a weak solution to the equation
    \begin{equation*}
        \partial_x\Bigl(h_{\infty}^3 \bigl(\partial_x^3 h_{\infty}-g\partial_x h_{\infty}\bigr) + M_{\infty}\frac{h_{\infty}^2}{(1+h_{\infty})^2}\partial_x h_{\infty}\Bigr) = 0
    \end{equation*}
    in the sense of Definition \ref{def:weak-solution}. 
    Moreover, \(h_{\infty}\) is even, periodic, non-decreasing in $[-\tfrac{\pi}{k_0},0)$, and \(h_{\infty}\geq 0\) with \(h_{\infty}(x) = 0\) precisely in \(x= \pm\tfrac{\pi}{k_0}+2j\tfrac{\pi}{k_0}\), \(j\in \Z\).
   The nonlinear flux
    \begin{equation*}
        h^3(\partial_x^3 h - g\partial_x h) + M\tfrac{h^2}{(1+h)^2}\partial_xh
    \end{equation*}
    can be continuously extended into the set $\{h_\infty = 0\}$ and vanishes everywhere.
\end{theorem}

\noindent\textbf{\textsc{Instability of periodic stationary states. }} Eventually, we discuss the stability of the periodic solutions close to their bifurcation point. This relies on the observation that for $M > M^*$ the constant steady-state solution $\bar{h} = 1$ to \eqref{eq:thin-film-equation} is spectrally unstable. By perturbation arguments we obtain spectral instability in $L^2(\R)$ of periodic solutions, which are sufficiently small. We refer to Section \ref{sec:stability} and Theorem \ref{thm:instability} for details.

\begin{theorem}[Instability]
   Let \(k_0>0\) and 
    \begin{equation*}
        \{(v(s),M(s)) : s \in (-\eps,\eps)\} \subset \Ucal \times \R
    \end{equation*}
    the bifurcation branch obtained in Theorem \ref{thm:main-global-bifurcation}. Then, there exists \(s_0>0\) such that \(h(s) = 1+v(s)\) is a $L^2$-spectrally unstable solution to \eqref{eq:thin-film-equation} for every \(|s|<s_0\).
\end{theorem}

\subsection{Physical background and modelling}\label{sec:modelling}

In this section, we present the derivation of equation \eqref{eq:thin-film-equation} using a lubrication approximation and starting from a Navier--Stokes system, see also \cite{nepomnyashchy2012} or \cite{nazaripoor2018}.

We consider an incompressible, viscous, Newtonian fluid resting on a impermeable heated plane. We assume that our film is homogeneous in one horizontal direction and that that the free surface is described by the graph of a function \(h=h(t,x)\) such that the fluid domain at time \(t>0\) is given by
\begin{equation*}
    \{(t,x,z) \, : \, x\in \R, \ 0 < z < h(t,x)\}.
\end{equation*}

Thermal variations in the fluid change its density due to buoyancy effects. In the Boussinesq approximation, one then models the density \(\rho(t,x)\) of the fluid to depend linearly on the temperature variations
\begin{equation*}
    \rho(t,x) = \rho_0 - \beta T(t,x),
\end{equation*}
where \(\beta >0\) is the coefficient of thermal expansion. Typically for thin films, these density variations are assumed to be very small and dominated by surface-tension effects, see e.g. \cite{vanhook1995}. Thus, in the following we will set \(\beta=0\). Hence, the density of the fluid is constant and we set \(\rho=\rho_0=1\). %This will also eliminate the buoyancy forces from the equations of motion. 

Under these assumptions, the equations of motion for the fluid are given by the Navier--Stokes system coupled to a transport-diffusion equation for the distribution of heat through the fluid
\begin{equation}\label{eq:Navier-Stoks-dim}
    \arraycolsep=2pt
    \def\arraystretch{1.2}
    \left\{
    \begin{array}{rcl}
        \partial_t u + (u\cdot\nabla)u & = & - \nabla p + \mu \Delta u - ge_z,  \\
        \div u & = & 0,   \\
        \partial_t T + (u\cdot \nabla) T & = & \chi \Delta T,
    \end{array}
    \right.
\end{equation}
with $t > 0$, $x \in \R$ and $0 < z < h(t,x)$. Here, \(u=(u_1(t,x,z),u_2(t,x,z))\) denotes the velocity field of the fluid, \(p=p(t,x,z)\) is the pressure, and \(T=T(t,x,z)\) the temperature. Furthermore, \(\mu>0\) is the kinematic viscosity, \(g>0\) the gravity of Earth, and \(\chi >0\) the thermal diffusivity. Eventually, \(e_z\) is the unit normal vector in \(z\)-direction.

We introduce the following boundary conditions for the fluid and the temperature. At the free surface, we assume a stress-balance condition and a kinematic boundary condition
\begin{equation}\label{eq:upper-boundary-fluid-dim}
    \arraycolsep=2pt
    \def\arraystretch{1.2}
    \left\{
    \begin{array}{rcl}
        \Sigma(p,u) n & = & \sigma \kappa n - (\partial_x \sigma) \tau, \\
        \partial_t h + u_1\partial_x h & = & u_2, 
    \end{array}
    \right.
\end{equation}
where the latter ensures that fluid particles on the surface stay on the surface. Here, \(\Sigma(p,u) = -p\Id + \mu \frac{\nabla u + \nabla u^T}{2}\) denotes the Cauchy stress tensor, \(\sigma=\sigma(T)\) denotes the surface tension, \(\kappa\) is the curvature of the free surface, \(\tau\) is the tangent vector, and \(n\) the outer unit normal at the free surface, which are given by
\begin{equation*}
    \kappa=\frac{\partial_x^2 h}{(1+|\partial_x h|^2)^{3/2}}, \quad \tau = \frac{1}{\sqrt{1+|\partial_x h|^2}}(1,\partial_xh), \quad n =  \frac{1}{\sqrt{1+|\partial_x h|^2}}(-\partial_xh,1).
\end{equation*}
Notice that the surface tension is not constant, but depends on the temperature.

At the solid bottom, we assume a no-slip boundary condition, that is
\begin{equation}\label{eq:lower-boundary-fluid-dim}
    u = 0, \quad t>0,\ x\in \R,\ z=0.
\end{equation}
For the temperature, we assume that at the surface holds
\begin{equation}\label{eq:upper-boundary-temp-dim}
    k \nabla T \cdot n = -K(T-T_g),\quad t>0,\ x\in \R,\ z=h(t,x),
\end{equation}
where \(k>0\) is the heat conductivity of the fluid, \(K>0\) is the heat exchange coefficient, and \(T_g\) the characteristic temperature of the surrounding gas. Moreover, at the solid bottom, we assume that
\begin{equation}\label{eq:lower-boundary-temp-dim}
    T = T_s,\quad t>0,\ x\in \R,\ z=0.
\end{equation}

Studying a thin fluid film, we assume that the surface profile varies on long spatial scales compared to the height of the film. We rescale the non-dimensional variables and perform a so-called lubrication approximation (or long-wave approximation). We assume that the typical horizontal length scale \(L=\tfrac{H}{\eps}\) with $0 < \eps \ll 1$ is very large compared to the vertical length scale \(H\). In order to account for the capillary dynamics, the time variable is rescaled of order \(\eps^2\). Then, the horizontal and vertical velocities are of order  \(\eps U_1\) and \(\eps^2 U_1\), respectively. The rescaled variables are
\begin{equation*}
\begin{split}
    \tilde{x} = \frac{\eps x}{H}, \quad \tilde{z} = \frac{z}{H}, \quad \tilde{t} = \frac{\eps^2}{t}{t_0}, \quad \tilde{u}_1 = \frac{u_1}{\eps U_1}, \quad \tilde{u}_2 = \frac{u_2}{\eps^2 U_1}, \quad \tilde{p} = \frac{p}{P}, \quad \tilde{T} = \frac{T}{\theta},\quad P = \frac{\mu U_1}{H},
\end{split}
\end{equation*}
where \(H,U_1,t_0,P\) and \(\theta\) are of order one. We also rescale \(\tilde{T}_s = \tfrac{T_s}{\theta}\) and \(\tilde{T}_g = \tfrac{T_g}{\theta}\).

Dropping the tildes, we find that the bulk equations \eqref{eq:Navier-Stoks-dim} in non-dimensional variables are given by 
\begin{equation}\label{eq:bulk-equations-non-dim}
    \arraycolsep=2pt
    \def\arraystretch{1.2}
    \left\{
    \begin{array}{rcll}
         \mathrm{Re}\eps^3(\partial_t u_1 + u_1\partial_x u_1 + u_2\partial_z u_1) & = & -\eps \partial_x p + \eps^3 \partial_x^2 u_1 +  \eps \partial_z^2 u_1, \\
         \mathrm{Re}\eps^4(\partial_t u_2 + u_1\partial_x u_2 + u_2 \partial_z u_2) & = & -\partial_z p + \eps^4 \partial_x^2 u_2 + \eps^2 \partial_z^2 u_2 - \mathrm{Ga}, \\
         \partial_x u_1 + \partial_z u_2 & = & 0,\\
         \eps^2 (\partial_t T + (u\cdot \nabla) T) & = & \frac{1}{\mathrm{Pr}} \left( \eps^2\partial_x^2 T + \partial_z^2 T\right)
    \end{array}
    \right.
\end{equation}
for \(t>0\), \(x\in \R\) and \(0<z<h(t,x)\).

We assume that the surface tension decreases linearly with temperature, that is there exists \(\alpha >0\) such that
\begin{equation*}
    \sigma = \sigma_0 - \alpha T.
\end{equation*}
Then, the boundary conditions \eqref{eq:upper-boundary-fluid-dim}--\eqref{eq:lower-boundary-temp-dim} in non-dimensional variables are given by
\begin{equation}\label{eq:boundary-conditions-non-dim}
\arraycolsep=2pt
\def\arraystretch{1.2}
\left\{
\begin{array}{rcll}
    -p + \frac{1}{1+\eps^2 |\partial_xh|^2} \bigl(\eps^2 |\partial_xh|^2 \partial_x u_1 - \eps^2\partial_xh(\partial_z u_1 + \eps^2 \partial_xu_2) + \eps^2 \partial_z u_2\bigr) & = & \frac{1}{\mathrm{Ca}}\frac{\partial_x^2h}{(1+\eps^2|\partial_xh|^2)^{\frac{3}{2}}}, \\
    \frac{1}{1+\eps^2|\partial_xh|^2}\bigl(\eps \partial_z u_1 + \eps^3(\partial_xu_2-|\partial_xh|^2 \partial_z u_1) + \eps^5 |\partial_xh|^2 \partial_xu_2\bigr) & = & -\eps \frac{\mathrm{Ma}}{\mathrm{Pr}} \partial_x T,\\ 
    \partial_t h + u_1 \partial_x h & = & u_2,\\
     \frac{1}{\sqrt{1+\eps^2|\partial_x h|^2}}\bigl(-\eps^2 \partial_x h\partial_x T + \partial_z T\bigr) & = & -\mathrm{Bi}(T-T_g), \\
\end{array}
\right.
\end{equation}
at the free boundary \(z=h(t,x)\). At the bottom $z=0$ we have that $u = 0$ and $T = T_s$.

In \eqref{eq:bulk-equations-non-dim} and \eqref{eq:boundary-conditions-non-dim}, the non-dimensional constants are defined as
\begin{equation*}
\begin{split}
    \text{Reynolds number: } & \mathrm{Re} = \frac{U_1 H}{\mu}, \qquad& 
    \text{Galileo number: }& \mathrm{Ga} = \frac{g H^2}{\mu U_1},\\
    \text{Prandtl number: } & \mathrm{Pr} = \frac{U_1H}{\chi}, &
    \text{Capillary number: } & \mathrm{Ca} = \frac{\mu U_1}{\sigma \eps^2},\\
    \text{Marangoni number: } & \mathrm{Ma} = \frac{2\alpha \theta H}{\mu\chi}, &
    \text{Biot number: } & \mathrm{Bi} = \frac{KH}{k}. \\
\end{split}
\end{equation*}

Furthermore, since we assume the capillary effects to dominate, the capillary number \(\mathrm{Ca}\) is of order one, that is $\sigma$ is of order $\varepsilon^{-2}$. Additionally, we assume that the Marangoni number is of order one, which implies that $\alpha$ is also of order one. Hence, variations of the surface tensions due to the temperature will be very small.

Now, we perform the lubrication approximation by keeping only the equations of lowest order in \(\eps\) in \eqref{eq:bulk-equations-non-dim} and \eqref{eq:boundary-conditions-non-dim}. That is, we obtain the following approximative system of equations for the bulk of the fluid
\begin{equation}\label{eq:bulk-non-dim}
    \arraycolsep=2pt
    \def\arraystretch{1.2}
    \left\{
    \begin{array}{rcll}
        \partial_x p & = & \partial_z^2 u_1, \\
        \partial_z p & = & -\mathrm{Ga}, \\
        \partial_x u_1 + \partial_z u_2 & = & 0, \\
        \partial_z^2T & = & 0
    \end{array}
    \right.
\end{equation}
and for the boundary conditions at the boundary \(z=h(t,x)\)
\begin{equation}\label{eq:boundary-non-dim}
    \arraycolsep=2pt
    \def\arraystretch{1.2}
    \left\{
    \begin{array}{rcll}
        -p & = & \frac{1}{\mathrm{Ca}} \partial_x^2 h, \\
        \partial_z u_1 & = & - \frac{\mathrm{Ma}}{\mathrm{Pr}} \partial_x T,\\
        \partial_t h + u_1\partial_x h & = & u_2, \\
        \partial_z T & = & - \mathrm{Bi}(T-T_g).
    \end{array}
    \right.
\end{equation}

Using the equations for the temperature and the boundary condition at the bottom, we find that 
\begin{equation*}
    T = T_s + \frac{\mathrm{Bi} (T_g - T_s)}{1 + \mathrm{Bi} h}z
\end{equation*}
and hence
\begin{equation*}
    \partial_x T|_{z=h} = - \frac{\mathrm{Bi}(T_s-T_g)}{(1+\mathrm{Bi}h)^2} \partial_x h.
\end{equation*}

Due to the incompressibility condition \eqref{eq:bulk-equations-non-dim}.3, the kinematic boundary condition \eqref{eq:boundary-non-dim}.3 can be written as
\begin{equation}\label{eq:integrated-kinematic}
    \partial_t h + \partial_x \int_{0}^{h} u_1 \de z = 0.
\end{equation}
In view of \eqref{eq:bulk-non-dim}.2 and \eqref{eq:boundary-non-dim}.1, the pressure is given by
\begin{equation*}
    p = - \mathrm{Ca}^{-1} \partial_x^2 h + \mathrm{Ga} (h-z).
\end{equation*}
Using the no-slip boundary condition at the solid bottom \eqref{eq:bulk-non-dim}.1 and \eqref{eq:boundary-non-dim}.2, the horizontal velocity is then given by
\begin{equation*}
    u_1 = \bigl(\mathrm{Ca}^{-1} \partial_x^3 h - \mathrm{Ga}\partial_x h\bigr)\bigl(hz - \tfrac{1}{2}z^2\bigr) + \frac{\mathrm{Ma}}{\mathrm{Pr}} \frac{\mathrm{Bi}(T_s-T_g)}{(1+\mathrm{Bi}h)^2} \partial_x h z.
\end{equation*}
Substituting this into \eqref{eq:integrated-kinematic} and evaluating the integral, we obtain the equation
\begin{equation*}
    \partial_t h + \partial_x \left(\tfrac{1}{3} h^3 \bigl(\mathrm{Ca}^{-1} \partial_x^3 h - \mathrm{Ga}\partial_x h\bigr)  +\frac{\mathrm{Ma}}{\mathrm{Pr}} \frac{\mathrm{Bi}(T_s-T_g)h^2}{(1+\mathrm{Bi}h)^2} \partial_x h  \right) = 0,
\end{equation*}
which is \eqref{eq:thin-film-equation} by setting \(\mathrm{Bi}=1\), \(\mathrm{Ca}^{-1} =3\), \(\mathrm{Ga}=3g\), and \(\tfrac{\mathrm{Ma}(T_s-T_g)}{\mathrm{Pr}}=M\).

\subsection{Related results}

We give a brief overview of related results in the mathematical and physical literature that are related to our problem.

\noindent\textbf{\textsc{Related results for thermocapillary thin-film models. }} While, to the best of our knowledge, there exist no previous rigorous analytical mathematical results on thermocapillary thin-film models, there is a vast literature on physical and numerical results for the model studied in this paper, as well as for many related models.

Equation \eqref{eq:thin-film-equation} has been studied numerically in \cite{thiele2004}. There, the same instability and bifurcating periodic solutions have been investigated numerically. The two-dimensional version of equation \eqref{eq:thin-film-equation} has been studied numerically in \cite{oron2000} and \cite{nazaripoor2018} for Newtonian fluids, and in \cite{mohammadtabar2022} for non-Newtonian fluids.

The amplitude equations which arise at the onset of instability in equation \eqref{eq:thin-film-equation} and which are derived in Section \ref{sec:stability} and Appendix \ref{sec:appendix} have been previously studied in \cite{nepomnyashchii1990,nepomnyashchy1994,nepomnyashchy2015,funada1987}.

%Moreover, there is a zoo of other models describing thin fluid films on heated surfaces. 
While this paper considers the deformational model given by equation \eqref{eq:thin-film-equation}, in different parameter regimes a non-deformational model has been derived in \cite{sivashinsky1982}. This models the physical system purely in terms of the temperature in a fixed fluid domain. A coupled model of equations for both the film height and temperature is derived by long-wave approximation in \cite{shklyaev2012} and \cite{samoilova2015}. In this coupled model, a Turing (or monotonic) and Turing--Hopf (or oscillatory) instability is observed. The emergence of the same instability in the full B\'enard-Marangoni problem has recently been verified by \cite{samoilova2014}. In \cite{batson2019}, a thin liquid film on a thick substrate has been studied numerically via a nonlocally coupled system. Problem \eqref{eq:thin-film-equation} has been studied for non-uniform heating in \cite{yeo2003}. Finally, there also exists a vast literature on multi-layer thermocapillary models out of which we mention \cite{nepomnyashchy1997,nepomnyashchy2007,nepomnyashchy2012a}.

For more details on the physics and modelling of thermocapillary fluid films, we refer the reader to the review paper \cite{oron1997}. We also mention the recent review \cite{witelski2020}. There, film rupture on hydrophobic surfaces is investigated. Notice that this model is very different to the model presented here as the main driving forces there are van der Waals forces.

\noindent\textbf{\textsc{Related results in global bifurcation. }} Global bifurcation has been a powerful tool to investigate patterns in many fluid models. 
We would like to point out that the thin-film equation \eqref{eq:thin-film-equation} can be written as a generalised Cahn--Hilliard model
\begin{equation*}
    \partial_t h + \partial_x\bigl(h^3 \partial_x(\partial_x^2 h- f(h))\bigr) = 0 \quad \text{with}\quad f(h) = gh - M\Bigl(\frac{1}{1+h} + \log\bigl(\frac{h}{1+h}\bigr)\Bigr),
\end{equation*}
see \eqref{eq:steady-state-equation}. Global bifurcation results for the Cahn--Hilliard equation have been previously obtained in \cite{kielhöfer1997,healey2015}.

Furthermore, fingering patterns in the Muskat problem, which models a two-layer fluid problem with a denser top fluid, have been obtained via a global bifurcation technique in \cite{ehrnström2013}.

Eventually, in water-wave models, global bifurcation results have been used to verify Whitham's conjecture of the existence of highest traveling waves with a sharp crest, see \cite{ehrnström2019}, and  \cite{bruell2021a, hildrum2023a} for related model equations.

\subsection{Outline of the paper}

The structure of the present paper is as follows: in Section \ref{sec:Hamiltonian} we introduce the second-order equation for the stationary solutions and derive its Hamiltonian structure. Via phase-plane analysis, we then classify the steady states. 

In Section \ref{sec:local}, we provide the functional analytic setup, perform a local bifurcation analysis, and include the asymptotics of the bifurcation curve close to the bifurcation point.

In Section \ref{sec:global} we use analytic global bifurcation theory to extend the bifurcation curve. Moreover, we rule out the possibility of the curve to be a closed loop and show that the curve approaches a state of film rupture.

We continue the analysis of the global bifurcation curve in Section \ref{sec:film-rupture}, where we provide uniform bounds on the global bifurcation curve and identify the limit points as periodic weak stationary solutions to \eqref{eq:thin-film-equation} exhibiting film rupture.

We study the stability and instability of the constant steady state in Section \ref{sec:stability} and show that at the onset of the bifurcation the bifurcation curve consists of spectrally unstable periodic solutions. Furthermore, we demonstrate that close to the critical Marangoni number $M^*$ a (conserved) long-wave instability occurs and the dynamics of \eqref{eq:thin-film-equation} is formally captured by a Sivashinsky equation. We give a formal derivation of this amplitude equation in Appendix \ref{sec:appendix}.

% \section{Well-posedness}

\section{Classification of steady states and the Hamiltonian system}\label{sec:Hamiltonian}

In this section, we study positive steady states of equation \eqref{eq:thin-film-equation}. First, we derive a second-order equation for the steady-state problem to \eqref{eq:thin-film-equation} and analyse its linearisation. Then we show that the second-order equation is a Hamiltonian system which we study via a phase-plane analysis to classify the stationary solutions to \eqref{eq:thin-film-equation}.

Stationary solutions to \eqref{eq:thin-film-equation} satisfy the fourth-order ordinary differential equation
\begin{equation*}
    \partial_x\Bigl(h^3(\partial_x^3 h - g\partial_x h) + M \dfrac{h^2}{(1+h)^2}\partial_x h \Bigr) = 0.
\end{equation*}
Integrating once with respect to $x$ yields
\begin{equation*}
    h^3(\partial_x^3 h - g\partial_x h) + M \dfrac{h^2}{(1+h)^2}\partial_x h = 0.
\end{equation*}
We set the integration constant to zero to keep the flat surface as an admissible steady state of \eqref{eq:thin-film-equation}. Since we look for positive steady states, we may divide by \(h^3\). Integrating once more, we obtain the second-order ordinary differential equation
\begin{equation}\label{eq:steady-state-equation}
    \partial_x^2 h = gh - M\Bigl(\frac{1}{1+h} + \log\Bigl(\frac{h}{1+h}\Bigr) \Bigr) + \tilde{K},
\end{equation}
where \(\tilde{K}\in \R\) is an integration constant. Notice that to each constant solution \(\bar{h}\) corresponds a different integration constant \(\tilde{K}=\tilde{K}(\bar{h};M)\) given by
\begin{equation*}
    \tilde{K}(\bar{h};M) := -g\bar{h} + M\Bigl(\frac{1}{1+\bar{h}} + \log\Bigl(\frac{\bar{h}}{1+\bar{h}}\Bigr) \Bigr).
\end{equation*}

%We first study bifurcations from constant steady states. Therefore, we choose 
Without loss of generality we set in the sequel \(\bar{h}=1\) and \(\tilde{K}(1;M)=-g + M\bigl(\tfrac{1}{2} + \log\bigl(\tfrac{1}{2}\bigr)\bigr)\). To study \eqref{eq:steady-state-equation} around the steady state \(h=\bar{h}\), we write \(h=\bar{h} + v\) and obtain the equation 
\begin{equation*}
    \partial_x^2 v = g(1+v) - M\Bigl(\frac{1}{2+v} + \log\Bigl(\frac{1+v}{2+v}\Bigr) \Bigr) + \tilde{K}.
\end{equation*}
By setting
\begin{equation}\label{eq:definition-K-scalar}
    K(\bar{v}) := \frac{1}{M}(\tilde{K}(1+\bar{v};M) - g), 
\end{equation}
for \(\bar{v}\in \R\), and thus \(K=K(0) =\tfrac{1}{2} + \log\bigl(\tfrac{1}{2}\bigr)\), we rewrite the equation above as
\begin{equation}\label{eq:shifted-steady-state-equation}
    \partial_x^2 v = gv - M\Bigl(\frac{1}{2+v} + \log\Bigl(\frac{1+v}{2+v}\Bigr) \Bigr) + MK.
\end{equation}
Equation \eqref{eq:shifted-steady-state-equation} can be written as a first-order system, which is given by
\begin{equation}
    \left\{\begin{split}
        \partial_x v &= w, \\
        \partial_x w &= gv - M\Bigl(\frac{1}{2+v} + \log\Bigl(\frac{1+v}{2+v}\Bigr) \Bigr) + MK.
    \end{split}\right.
    \label{eq:first-order-system}
\end{equation}
To study bifurcations of the trivial state, we linearise the system about $(v,w) = (0,0)$ and obtain
\begin{equation*}
    L = \begin{pmatrix}
        0 & 1 \\
        g - \frac{1}{4}M & 0
    \end{pmatrix}.
\end{equation*}
The eigenvalues of $L$ are given by 
\begin{equation*}
    \lambda_\pm = \pm \sqrt{g - \frac{1}{4}M}.
\end{equation*}
If $M < M^\ast := 4g$ the eigenvalues are real, where one is positive and one negative. At $M = M^\ast$ the eigenvalues collide at $\lambda = 0$ and split into two complex conjugated, purely imaginary eigenvalues for $M > M^\ast$.

\begin{figure}
    \centering
    \includegraphics[width = 0.9\textwidth]{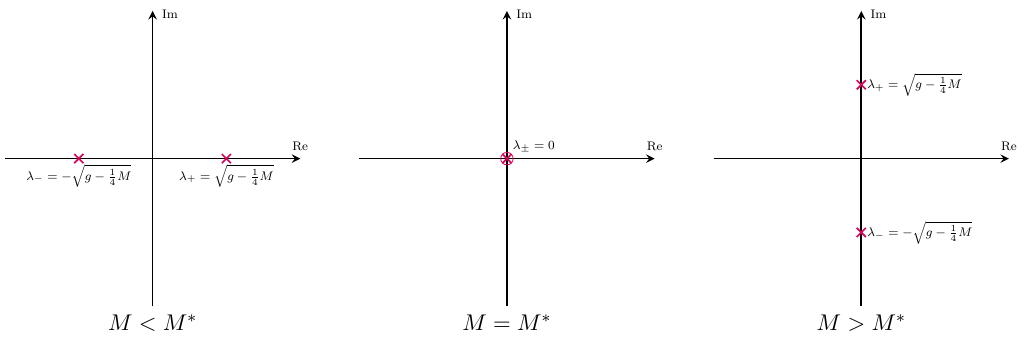}
    \caption{Plot of the different spectral situations. For $M < M^*$, the matrix $L$ has two real eigenvalues. These eigenvalues collide at zero if $M = M^*$ and then form a pair of purely imaginary, complex conjugated eigenvalues for $M > M^*$.}
    \label{fig:eigenvalues-bifurcation}
\end{figure}
% \begin{center}
%     Add nice picture for three cases.
% \end{center}

This suggests that for $M < M^\ast$ a homoclinic orbit bifurcates from the trivial state, that is, a solution $(v_\mathrm{hom},w_\mathrm{hom})$ of \eqref{eq:shifted-steady-state-equation} satisfying
\begin{equation*}
    \lim_{\vert x \vert \rightarrow \infty} (v_\mathrm{hom},w_\mathrm{hom})(x) = (0,0).
\end{equation*}
For $M > M^\ast$, we expect the bifurcation of periodic solutions with period close to $\frac{2\pi}{\vert \lambda_\pm \vert}$.
This is indeed the case as the following analysis shows.

The main observation to understand the bifurcation structure of \eqref{eq:shifted-steady-state-equation} is that \eqref{eq:first-order-system} is a Hamiltonian system with Hamiltonian
\begin{equation}\label{eq:Hamiltonian}
    \Hcal(v,w) = \frac{1}{2}w^2 - \frac{g}{2} v^2 + M(1+v)\log\Bigl(\frac{1+v}{2+v}\Bigr) - MKv
\end{equation}
on the phase space \((v,w) \in \Omega:=(-1,\infty)\times \R\). Notice that we require \(v>-1\) since we are interested in positive solutions for \(h=1+v\). However, the Hamiltonian is also defined on the boundary of phase space \(v=-1\).

The Hamiltonian system plays a crucial role in the remainder of the paper. Therefore, we collect some results about its dynamics under the more general assumption that 
\begin{equation}\label{eq:constant-assumption}
    K \leq K(0) = \frac{1}{2} + \log\bigl(\frac{1}{2}\bigr)
\end{equation}
 in the following lemma.
This more general assumption implies that if there are two fixed points \((v_l,w_l)\), \((v_u,w_u)\) of the Hamiltonian system, then it must hold \(v_l\leq 0 \leq v_u\) and we have \(v_l=v_u=0\) if and only if \(K = K(0)\) and \(M=M^*\). We will need the more general choice of \(K\leq K(0)\) later when studying the global bifurcation curves of the periodic solutions, see the proof of Proposition \ref{prop:film-rupture}.

\begin{figure}[H]
    \centering
    \includegraphics[width=0.5\textwidth]{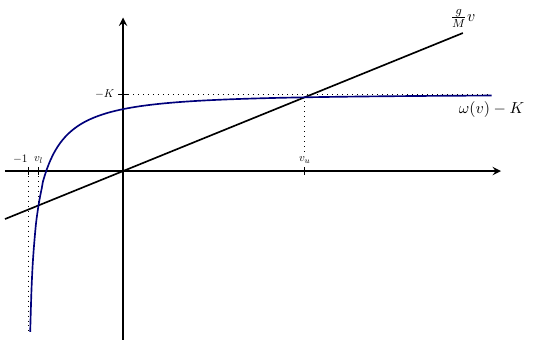}
    \caption{Graphical representation of the fixed points of the Hamiltonian system \eqref{eq:Hamiltonian}, which lie exactly at the intersection of \(\omega(v)-K\) and the line \(\tfrac{g}{M}v\).}
    \label{fig:fixed-points}
\end{figure}

\begin{lemma}\label{lem:orbits-Hamiltonian-system}
    If $K$ satisfies assumption \eqref{eq:constant-assumption}, the Hamiltonian dynamics generated by \eqref{eq:Hamiltonian} satisfies the following statements.
    \begin{enumerate}[label=(\arabic*)]
        \item\label{it:orbit-1} There are two fixed points $(v_l,w_l), (v_u,w_u)\in  \Omega$ with $w_l=w_u = 0$ and $-1 < v_l \leq 0 \leq v_u$, where equality \(v_l=v_u=0\) holds if and only if $K = K(0)$ and $M = M^\ast$.
        \item\label{it:orbit-2} If $v_l < v_u$ the lower fixed point $(v_l,0)$ is a minimum of $\Hcal$ and there is a neighborhood $U$ of $(v_l,0)$ in $\Omega$, which is filled with periodic orbits.
        \item\label{it:orbit-3} If $v_l < v_u$ the upper fixed point $(v_u,0)$ is a saddle point and if additionally $\Hcal(-1,0) \geq \Hcal(v_u,0)$ there exists a homoclinic orbit $(v_\mathrm{hom},w_\mathrm{hom})$ satisfying
        \begin{equation*}
            \lim_{\vert x \vert \rightarrow +\infty} (v_\mathrm{hom},w_\mathrm{hom}) = (v_u,0).
        \end{equation*}
        The orbit $\{(v_\mathrm{hom}(x),w_\mathrm{hom}(x)) \in \Omega \,:\, x \in \R\}$ is the boundary of $U$. If $\Hcal(-1,0) < \Hcal(v_u,0)$, then the boundary of \(U\) consists of the level set of \(\Hcal(-1,0)\).
    \end{enumerate}
\end{lemma}

\begin{figure}[H]
    \centering
    \begin{minipage}[t]{0.49\textwidth}
            \includegraphics[width=1\textwidth]{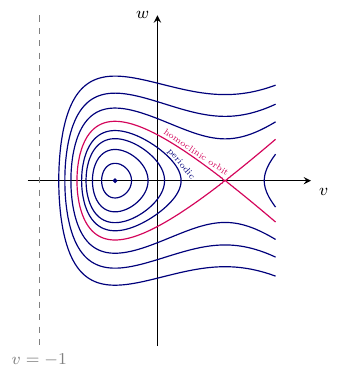}
    \end{minipage}
    \begin{minipage}[t]{0.49\textwidth}
            \includegraphics[width=1\textwidth]{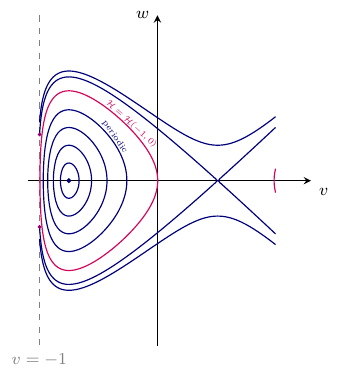}
    \end{minipage}
    \caption{Plot of the phase portrait for the Hamiltonian dynamics \eqref{eq:Hamiltonian}. On the left, the periodic orbits are contained by the homoclinic orbit, while on the right the periodic orbits are contained by the level set of $\Hcal(-1,0)$ and no homoclinic orbit exists. In particular, all level sets of outside of the level set of $\Hcal(-1,0)$ intersect $v=-1$ and cannot form periodic orbits.}
    \label{fig:phase-space}
\end{figure}
% \begin{center}
%     Add picture of phase space.
% \end{center}

\begin{proof}
    We first note that $(v,w)$ is a fixed point of the Hamiltonian dynamics generated by $\Hcal$ if $\nabla \Hcal(v,w) = 0$. Therefore, we obtain $w = 0$ and
    \begin{equation}\label{eq:fixed-point}
        \frac{g}{M} v = \frac{1}{2+v} + \log\Bigl(\frac{1+v}{2+v}\Bigr) - K =: \omega(v) - K.
    \end{equation}
    Then $v \mapsto \omega(v)$ is strictly concave and $\omega(0) - K \geq 0$ due to \eqref{eq:constant-assumption}. Since the left-hand side is a linear function, which vanishes at $v = 0$, \eqref{eq:fixed-point} has precisely two solutions $v_l < 0 < v_u$ if $\omega(0) - K > 0$ and thus $K < K(0)$. If $K = K(0)$, \eqref{eq:fixed-point} has at least one solution $v=0$. This is a double root if and only if $v \mapsto \frac{g}{M}v$ is tangential to $v \mapsto \omega(v) - K(0)$ at $v = 0$. Since \(\frac{d}{dv}\omega(0) = \tfrac{1}{4}\), this holds true if and only if \(M=M^* = 4g\). If $M < M^\ast$, we have \(\frac{g}{M} > \frac{d}{dv}\omega(0)\) and hence there is an additional solution $v_l < 0 = v_u$. Moreover, if $M > M^\ast$, it holds \(\frac{g}{M} < \frac{d}{dv}\omega(0)\) and there is an additional solution $v_l = 0 < v_u$. This proves \ref{it:orbit-1}.

    To show \ref{it:orbit-2} and \ref{it:orbit-3} we first note that the Hamiltonian $\Hcal$ decomposes as $\Hcal(v,w) = \Fcal(w) + \Gcal(v)$ with
    \begin{equation}\label{eq:Gcal}
        \Fcal(w) = \frac{1}{2} w^2, \quad \Gcal(v) = -\frac{g}{2} v^2 + M(1+v)\log\Bigl(\frac{1+v}{2+v}\Bigr) - MKv.
    \end{equation}
    Therefore, the eigenvalues of $D^2\Hcal(v_0,w_0)$ are given by $\Fcal^{\prime\prime}(w_0)$ and $\Gcal^{\prime\prime}(v_0)$. Since $\Fcal^{\prime\prime}(w_0) = 1$, it is sufficient to determine the sign of $\Gcal^{\prime\prime}(v_0)$ given by
    \begin{equation*}
        \Gcal^{\prime\prime}(v_0) = -g+\frac{M}{(v_0+1) (v_0+2)^2}.
    \end{equation*}
    We have $\Gcal^{\prime\prime}(0) = -g + \frac{1}{4}M$ and for fixed \(M>0\), the mapping $v_0 \mapsto \Gcal^{\prime\prime}(v_0)$ is monotonically decreasing. In addition, assuming that $v_l < v_u$, \ref{it:orbit-1} implies that $v \mapsto \Gcal(v)$ has two extrema in the interior of $(-1,\infty)$ with $v_l \leq 0 \leq v_u$. 
    
    For $M > M^\ast$ it holds that $\Gcal^{\prime\prime}(0) > 0$ and since $v_l \leq 0$ we conclude that \(\Gcal^{\prime\prime}(v_l) > 0\). So, \(v_l\) is a minimum of \(\Hcal\). Moreover, \(v_u\) is a maximum of \(\Gcal\) since \(\Gcal(v) \to -\infty\) as \(v\to +\infty\), so \((v_u,0)\) is a saddle point of \(\Hcal\).

    For \(M< M^\ast\), it holds that $\Gcal^{\prime\prime}(0) < 0$. Since \(v_u \geq 0\), we find that \(\Gcal^{\prime\prime}(v_u) < 0\) by using again the monotonicity of \(\Gcal^{\prime\prime}\). This shows that \((v_u,0)\) is a saddle point of \(\Hcal\). Furthermore, \(v_l\) is a minimum of \(\Gcal\), hence \((v_l,0)\) is a minimum for \(\Hcal\).

    If \(M=M^*\) and \(K<K(0)\), we have \(\Gcal^{\prime\prime}(0) = 0\) and \(v_l< 0\) and \(v_u>0\). By the strict monotonicity of \(\Gcal^{\prime\prime}\), we thus find that \(\Gcal^{\prime\prime}(v_l) > 0 > \Gcal^{\prime\prime}(v_u)\). Again, \((v_l,0)\) is a minimum of \(\Hcal\) and \((v_u,0)\) is a saddle point.

    Therefore, if $v_l < v_u$ we find that $(v_l,0)$ is a minimum of $\Hcal$ and $(v_u,0)$ is a saddle point of $\Hcal$. In fact, $(v_l,0)$ is an isolated minimum since $\tfrac{d^4}{dv^4}\Gcal(v_l) > 0$.

    Since $(v_l,0)$ is an isolated minimum of $\Hcal$, there exists a neighborhood $U$ of $(v_l,0)$ such that every point in $U$ belongs to a level set of $\Hcal$, which forms a closed curve around $(v_l,0)$ and does not contain a fixed point of $\Hcal$. Since orbits of Hamiltonian systems are confined to the level sets of the Hamiltonian, the neighborhood $U$ is filled with periodic orbits. This proves \ref{it:orbit-2}. In view of $\Gcal$ having precisely two extrema, the boundary of $U$ is given by the level set of the minimum of $\Hcal(-1,0)$ and $\Hcal(v_u,0)$.

    Finally, if $\Hcal(-1,0) \geq \Hcal(v_u,0)$ the level set of $\Hcal$ containing $(v_u,0)$ contains a closed curve, which loops around $(v_l,0)$. Therefore, the Hamiltonian system has a homoclinic orbit to $(v_u,0)$. Since this closed curve is also the boundary of $U$ part \ref{it:orbit-3} is proved.
\end{proof}

\begin{remark}
    If $\Hcal(-1,0) \geq \Hcal(v_u,0)$ and there exists a homoclinic orbit \((v_\mathrm{hom},w_\mathrm{hom})\), we find that there is a sequence of periodic orbits in \(U\) approaching the homoclinic orbit such that their periods converge to \(+\infty\). This is known in the literature as \emph{blue-sky catastrophe} or \emph{homoclinic period blow-up}, c.f. \cite{devaney,vanderbauwhede1992}.
\end{remark}

The existence result for solutions to the Hamiltonian systems in the previous Lemma \ref{lem:orbits-Hamiltonian-system} implies in particular the existence of solutions to the integrated equation \eqref{eq:shifted-steady-state-equation} and thus, solutions to the thin-film equation \eqref{eq:thin-film-equation}.

The previous analysis demonstrates the existence of both solitary solutions, i.e.\ solutions which converge to the same fixed state for $x \rightarrow \pm \infty$, corresponding to homoclinic orbits, and infinitely many periodic solutions. In particular, for $M < M^\ast$ and $K = K(0)$ we find a homoclinic orbit to the steady state $(v_u,0) = (0,0)$. This establishes the following solution to the thin-film equation \eqref{eq:thin-film-equation}.

\begin{corollary}\label{cor:homoclinic} 
    Fix $\bar{h} = 1$. Then, there exists $M^\ast = 4g$ such that for all $M < M^\ast$, the thin-film equation \eqref{eq:thin-film-equation} has stationary solutions of the form $h(x) = \bar{h} + v(x)$ with $v(x) < 0$ for all $x \in \R$ and
    \begin{equation*}
        \lim_{\vert x \vert \rightarrow \infty} v(x) = 0.
    \end{equation*}
\end{corollary}

For \(M>M^\ast\), we see that the structure of the orbits changes. There is still a homoclinic orbit, but now about a fixed point \(v_u >0\). Meanwhile, the dynamics close to the fixed point \(v_l=0\) is filled by periodic orbits. In this paper, we will be concerned with a detailed analysis of these periodic orbits. The analysis of the homoclinic orbit found in Corollary \ref{cor:homoclinic} will be left for future studies.

Since \eqref{eq:thin-film-equation} is in divergence form, the mass of the solution is formally conserved. Therefore, it makes physical sense to search for a bifurcating family of periodic solutions, which conserve their mass over a period along the bifurcation curve. This is the main topic of the following sections \ref{sec:local} to \ref{sec:film-rupture}.

\section{Functional analytic setup and local bifurcation}\label{sec:local}

\subsection{Functional analytic setup}

As we have seen in the previous section, for \(M>M^*\) there exist periodic solutions close to the fixed point \(v_l=0\). The structure of these periodic solutions will be studied using bifurcation theory.

Since \eqref{eq:first-order-system} has a Hamiltonian structure, we may analyse the period of periodic orbits bifurcating from \(v_l = 0\) for $M > M^\ast$ using the Lyapunov subcentre theorem, see \cite[Thm. 4.1.8]{schneider2017}. In fact, the wave number is close to $\vert\Im(\lambda_\pm)\vert$, where
\begin{equation*}
    \lambda_\pm = \pm\sqrt{g - \frac{1}{4}M}
\end{equation*}
are the eigenvalues of the linearisation of \eqref{eq:first-order-system} about $(v,w) = (0,0)$. Then their period is close to $\frac{2\pi}{\vert\Im(\lambda_\pm)\vert}$ and since $M > M^\ast = 4g$, the wave number is close to
\begin{equation*}
    k^\ast_M = \sqrt{\frac{1}{4}M-g}.
\end{equation*}
In particular, $k^\ast_M$ vanishes for $M \rightarrow M^\ast$ and $k^\ast_M \rightarrow \infty$ as $M \rightarrow \infty$. Therefore, for every wave number $k > 0$ exists a $M^\ast(k)=M^* + 4k^2$ such that for $M > M^\ast(k)$ periodic solutions with wave number $k$ close to $(v_u,0)$ exist in \eqref{eq:first-order-system}.

In light of this structure, we now fix a wave number $k_0 > 0$ and study the bifurcation of periodic solutions of \eqref{eq:thin-film-equation} with wave number $k_0$, which conserve mass along the bifurcation curve in the sense that
\begin{equation*}
    \frac{k_0}{2\pi} \int_{-\frac{\pi}{k_0}}^{\frac{\pi}{k_0}} h(x) \de x  =: \fint_{-\frac{\pi}{k_0}}^{\frac{\pi}{k_0}} h(x) \de x = \fint_{-\frac{\pi}{k_0}}^{\frac{\pi}{k_0}} \bar{h} \de x = 1.
\end{equation*}

%We therefore work in the following framework.
Writing again \(h=\bar{h} + v\), we define
\begin{equation*}
    F(v,M) := \partial_x^2 v - gv + M\Bigl(\frac{1}{2+v} + \log\Bigl(\frac{1+v}{2+v}\Bigr) \Bigr) - MK(v) =: \partial_x^2 v + f_M(v),
\end{equation*}
where 
\begin{equation*}
    K(v) := \fint_{-\frac{\pi}{k_0}}^{\frac{\pi}{k_0}} \frac{1}{2+v} + \log\Bigl(\frac{1+v}{2+v}\Bigr) \de x
\end{equation*}
This choice of $K$ guarantees that the mass of $v$ vanishes over one period, since $F(\cdot, M)$ leaves the space of functions with vanishing integral mean invariant. Furthermore, observe that this is an extension of \(K\) introduced in \eqref{eq:definition-K-scalar} to non-constant functions.

Next, we introduce the function spaces
\begin{equation*}
    \begin{split}
        \Xcal &= \Bigl\{v \in H^2_\mathrm{per} \,:\, \fint_{-\frac{\pi}{k_0}}^{\frac{\pi}{k_0}} v(x) \de x = 0 \text{ and } v \text{ is even}\Bigr\}, \\
        \Ycal &= \Bigl\{v \in L^2_\mathrm{per} \,:\, \fint_{-\frac{\pi}{k_0}}^{\frac{\pi}{k_0}} v(x) \de x = 0 \text{ and } v \text{ is even}\Bigr\}.
    \end{split}
\end{equation*}
Here, for a given wave number \(k_0\), we denote by \(H^2_{\mathrm{per}}\) the space of \(\tfrac{2\pi}{k_0}\)-periodic functions in \(H^2_\loc(\R)\), and \(L^2_{\mathrm{per}}\) the space of \(\tfrac{2\pi}{k_0}\)-periodic \(L^2_{\loc}\)-functions. The space \(\Xcal\) is equipped with the norm of \(H^2_{\mathrm{per}}\), that is the \(H^2\)-norm on \(\bigl(-\tfrac{\pi}{k_0},\tfrac{\pi}{k_0}\bigr)\), and \(\Ycal\) with the \(L^2\)-norm on \(\bigl(-\tfrac{\pi}{k_0},\tfrac{\pi}{k_0}\bigr)\).

We define the open subset \(\Ucal\) of \(\Xcal\) by
\begin{equation*}
    \Ucal := \{v\in \Xcal \,:\, v > -1\}.
\end{equation*}
In view of $h = v+1$, this is the set of all surface profiles, which have strictly positive height. Then, \(F\) is a well-defined map from \(\Ucal \times \R\) to \(\Ycal\) and we consider the bifurcation problem
\begin{equation}\label{eq:bifurcation-problem}
    F(v,M) = 0.
\end{equation}

\subsection{Local bifurcation}\label{sec:local-bifurcation}

We first establish the existence of a local bifurcation curve emerging at $M = M^\ast(k_0) = M^* + 4k_0^2$. The proof relies on an application of the Crandall--Rabinowitz theorem on the bifurcation from a simple eigenvalue.

\begin{theorem}[Local bifurcation]\label{thm:local-bifurcation}
    Fix \(k_0 >0\). Then at \((0,M^*(k_0))\) a subcritical pitchfork bifurcation occurs and there exist $\eps >0$ and a branch of solutions
    \begin{equation*}
        \{(v(s),M(s)) : s \in (-\eps,\eps)\} \subset \Ucal \times \R
    \end{equation*}
    to the bifurcation problem \eqref{eq:bifurcation-problem} with expansions
    \begin{equation}\label{eq:expansions}
        \begin{split}
            v(s) &= s \cos(k_0 x) + \tau(s), \\
            M(s) &= M^\ast(k_0) - \frac{(g+k_0^2)(8g+41k_0^2)}{12k_0^2} s^2 + \Ocal(\vert s \vert^3),
        \end{split}
    \end{equation}
    where $\tau = \Ocal(\vert s \vert^2)$ in $\Xcal$.
\end{theorem}

\begin{proof}
    We first show the existence of a local bifurcation branch using the Crandall--Rabinowitz theorem \cite[Thm. 8.3.1]{buffoni2003}. Since, \(F(0,M)=0\) for all \(M\in \R\), the following conditions need to be satisfied: 
    \begin{enumerate}[label=(\arabic*)]
        \item\label{it:local-1} \(L:=\partial_v F(0,M^\ast(k_0))\) is a Fredholm operator of index zero;
        \item\label{it:local-2} \(\ker L = \spn\{\xi_0\}\), \(\xi_0 \in \Xcal\setminus\{0\}\) is one-dimensional;
        \item\label{it:local-3} transversality condition: \(\partial_{v,M}^2 F(0,M^\ast(k_0))(\xi_0,1) \notin \ran L\).
    \end{enumerate}
    We first calculate
    \begin{equation*}
        L = \partial_v F(0,M^\ast(k_0)) = \partial_x^2 + \Bigl(\frac{1}{4}M^\ast(k_0) - g\Bigr) = \partial_x^2 + k_0^2,
    \end{equation*}
    where the integral term $\partial_v K(v)$ vanishes since $\Xcal$ contains only functions with mean zero. The Fourier symbol of $L$ is given by
    \begin{equation*}
        \hat{L}(\ell) = k_0^2 (1-\ell^2), \quad \ell \in \N \setminus \{0\},
    \end{equation*}
    where $\ell \neq 0$, since all elements of $\Xcal$ have vanishing mean. Additionally, because $\Xcal$ contains only even functions, we  restrict to $\ell \geq 1$. Hence, the kernel of $L$ is given by
    \begin{equation*}
        \ker L = \spn(\cos(k_0 \cdot))
    \end{equation*}
    and the range of \(L\) is given by
    \begin{equation*}
        \ran L = \spn\{\cos(\ell k_0 \cdot) : \ell \geq 2\}.
    \end{equation*}
    This demonstrates that \(L\) is Fredholm with index zero and that the kernel of \(L\) is one-dimensional, spanned by \(\xi_0=\cos(k_0 \cdot)\).

    The transversality condition \ref{it:local-3} is satisfied, since
    \begin{equation*}
        \partial^2_{v,M} F(0,M^\ast(k_0))(\xi_0,1) = \frac{1}{4}\xi_0 \notin \ran L.
    \end{equation*}

    By the Crandall--Rabinowitz theorem, this proves the existence of a local bifurcation branch consisting of \(\frac{2\pi}{k_0}\)-periodic solutions with mean zero
    \begin{equation*}
        \{(v(s),M(s)) : s \in (-\eps,\eps)\} \subset \Ucal \times \R
    \end{equation*}
    bifurcating from the constant steady state at \(M^\ast(k_0)\).

    To derive the expansions \eqref{eq:expansions} we apply \cite[Cor. I.5.2]{kielhöfer2012}, and find that
    \begin{equation*}
        v(s) = s\xi_0 + \tau(s),
    \end{equation*}
    where \(\tau(s)=\Ocal(|s|^2)\) in \(\Xcal\). Moreover, in view of \cite[Eqs. I.6.3 and I.6.11]{kielhöfer2012}, we have
    \begin{equation*}
        \dot{M}(0) = 0, \qquad \ddot{M}(0) = - \frac{(g+k_0^2)(8g+41k_0^2)}{6k_0^2}.
    \end{equation*}
    Since \(\ddot{M}(0) < 0\) and \(\dot{M}(0)=0\), we  conclude that the occurring bifurcation is a subcritical pitchfork bifurcation. 
    %This concludes the proof.
\end{proof}

\begin{remark}
    Note that the assumption that elements of the function space \(\Xcal\) are even removes the translational invariance of the thin-film equation. This is necessary to guarantee that the kernel is one-dimensional. In fact, all translations of the functions on the bifurcation branch are also periodic solutions to \eqref{eq:thin-film-equation} with mean zero.
\end{remark}

Using the expansions \eqref{eq:expansions} for $(v,M)$ on the bifurcation branch with respect to $s$, we  derive the following expansion for $v$ in terms of $M$. Denoting by \(\delta^2=M^\ast(k_0)-M(s)\), resolving the identity
\begin{equation*}
    \delta^2 = \frac{(g+k_0^2)(8g+41k_0^2)}{12k_0^2} s^2 + \Ocal(\vert s \vert^3)
\end{equation*}
for \(s\) and inserting this into the expansion for \(v\), we obtain the following corollary.

\begin{corollary}\label{cor:expansion}
    For any $k_0 > 0$ exists $\delta_0 > 0$ such for all $\delta \in (0,\delta_0)$ the expansion
    \begin{equation}\label{eq:expansion-bifurcation}
        v(x) = \pm \delta \sqrt{\frac{12k_0^2}{(g+k_0^2)(8g+41k_0^2)}} \cos(k_0 x) + \Ocal(\delta^2)
    \end{equation}
    holds, where again the remainder is small in \(\Xcal\).
\end{corollary}

\section{Global bifurcation}\label{sec:global}

In this section, the local bifurcation curve found in Theorem \ref{thm:local-bifurcation} is extended globally by using analytic global bifurcation theory. Eventually, we study the behaviour of the global bifurcation curve: nodal properties are used to rule out the possibility of the branch to form a closed loop. Furthermore and heavily relying on the Hamiltonian structure of \eqref{eq:steady-state-equation}, we show that the minimum of the solutions along the bifurcation branches approaches $-1$, that means asymptotically the solutions on the global bifurcation branch approach a state of film rupture.

\begin{theorem}[Global bifurcation]\label{thm:global-bifurcation}
    Let \(k_0>0\) and 
    \begin{equation*}
        \{(v(s),M(s)) : s \in (-\eps,\eps)\} \subset \Ucal \times \R
    \end{equation*}
    the bifurcation branch obtained in Theorem \ref{thm:local-bifurcation}. Then, there exists a globally defined continuous curve
    \begin{equation*}
        \{(v(s),M(s)) : s \in \R\} \subset \Ucal \times \R
    \end{equation*}
    consisting of smooth solutions to the bifurcation problem \eqref{eq:bifurcation-problem}, which is bounded in \(\Xcal \times (0,\infty)\) and such that
    \begin{equation*}
        \inf_{s\in \R} \min_{x\in \bigl[-\tfrac{\pi}{k_0},\tfrac{\pi}{k_0}\bigr]}v(s) = -1.
    \end{equation*}
\end{theorem}

The proof of Theorem \ref{thm:global-bifurcation} consists of multiple steps. First, using \cite[Thm. 9.1.1]{buffoni2003}, we construct a global extension of the local bifurcation curve that may be unbounded, a closed loop or approach the boundary of the phase space, c.f. Proposition \ref{prop:global-bifurcation}. We will then proceed to rule out that it is a closed loop in Proposition \ref{prop:no-closed-loop}. Finally, we will obtain uniform bounds for \((v(s),M(s))\) in \(\Xcal\times (0,\infty)\) in Lemma \ref{lem:uniform-bound-M} and Proposition \ref{prop:uniform-bound-H2}.

The results of Theorem \ref{thm:global-bifurcation} are also found in a numerical treatment of the problem using a numerical continuation algorithm. These findings are represented in the following Figure \ref{fig:numerical-continuation-plots}.

\begin{figure}[H]
    \centering
        \includegraphics[width=0.8\textwidth]{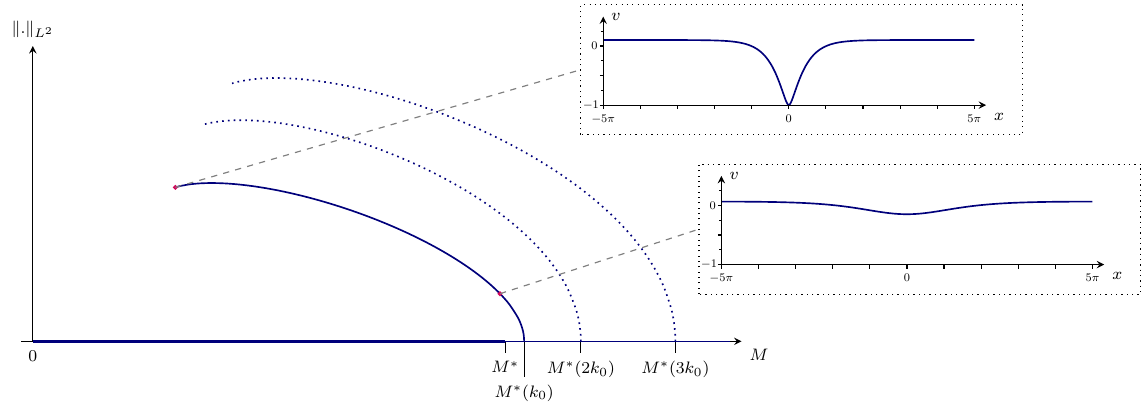}
    \caption{Numerical plot of the global bifurcation diagram and the associated solution plots at selected points on the branch. Notice that for illustration purposes, we plot the periodic solutions shifted by half a period (otherwise the minimum would occur at the boundary). The bifurcation curve is given for $g = 1$ on a domain $[-5\pi,5\pi]$ using Neumann boundary conditions. The dotted lines mark the secondary and tertiary bifurcation branches that emerge at \(M^*(2k_0)\) and \(M^*(3k_0)\) and contain solutions with half and one-third period.  The continuation was performed with pde2path, see \cite{uecker2014,uecker2021} and \cite{zotero-3112}.}
    \label{fig:numerical-continuation-plots}
\end{figure}

\subsection{Existence of a global bifurcation branch}\label{sec:global-bifurcation}

Next, we show that there exists a global extension of the local bifurcation branch established in Theorem \ref{thm:local-bifurcation}. To achieve this, we use analytic global bifurcation theory \cite{buffoni2003}.

\begin{proposition}\label{prop:global-bifurcation}
    Let \(k_0>0\) and 
    \begin{equation*}
        \{(v(s),M(s)) : s \in (-\eps,\eps)\} \subset \Ucal \times \R
    \end{equation*}
    the bifurcation branch obtained in Theorem \ref{thm:local-bifurcation}. Then, there exists a globally defined continuous curve
    \begin{equation*}
        \{(v(s),M(s)) : s \in \R\} \subset \Ucal \times \R
    \end{equation*}
    consisting of smooth solutions to the bifurcation problem \eqref{eq:bifurcation-problem} and extending the local bifurcation branch such that at least one of the following conditions is satisfied:
    \begin{enumerate}[label=(C\arabic*)]
        \item\label{it:global-condition-1} \(\|(v(s),M(s))\|_{\Xcal\times \R} \longrightarrow +\infty\) as \(s\to \pm \infty\);
        \item\label{it:global-condition-2} \((v(s),M(s))\) approaches the boundary of \(\Ucal\);
        \item\label{it:global-condition-3} \((v(s),M(s))\) is a closed loop.
    \end{enumerate}
\end{proposition}

\begin{remark}
    Notice that the conditions \ref{it:global-condition-1}-\ref{it:global-condition-3} at infinity in Proposition \ref{prop:global-bifurcation} are not mutually exclusive. This can for example be seen in \cite{ehrnström2019}, where both \ref{it:global-condition-1} and \ref{it:global-condition-2} occur simultaneously. The bifurcation branch approaches the boundary and looses regularity, which leads to a norm blow-up.
\end{remark}

\begin{proof}
    To prove Proposition \ref{prop:global-bifurcation}, we use \cite[Thm. 9.1.1]{buffoni2003}. We already know that \((0,M)\in \Ucal\) for all \(M\in \R\), that \(\dot{M}(s) \not\equiv 0\) on \((-\eps,\eps)\) by \eqref{eq:expansions} and we have checked condition (G3) in \cite[Ch. 9]{buffoni2003} by verifying \ref{it:local-2} and \ref{it:local-3} in the proof of Theorem \ref{thm:local-bifurcation}. Hence, it remains to check the following conditions.
    \begin{enumerate}[label=(\arabic*)]
        \item\label{it:global-1} \(\partial_v F(v,M)\) is a Fredholm operator of index zero, whenever \((v,M)\in \Ucal\times \R\) is in the set
        \begin{equation*}
            \Scal:=\{(v,M) \in \Ucal \times \R \,:\, F(v,M) = 0\}
        \end{equation*}
        of solutions to the bifurcation problem \eqref{eq:bifurcation-problem};
        \item\label{it:global-2} all bounded closed subsets of the solution set \(\Scal\) are compact in \(\Xcal\times \R\).
    \end{enumerate}
    To check \ref{it:global-1}, we compute
    \begin{equation*}
        \partial_v F(v,M)v_0 = \partial_x^2v_0 + \Bigl(M\frac{1}{(v+1)(v+2)^2}-g\Bigr)v_0 - M \fint_{-\frac{\pi}{k_0^2}}^{-\frac{\pi}{k_0^2}} \frac{1}{(v+1)(v+2)^2}v_0 \de x.
    \end{equation*}
    Note that this is a nonlocal operator. Since \(\partial_x^2 \colon \Xcal \to \Ycal\) is invertible, we may write
    \begin{equation*}
        \partial_v F(v,M) = \partial_x^2 \partial_x^{-2}\partial_v F(v,M)
    \end{equation*}
    and obtain that
    \begin{equation*}
        \tilde{L}_{v,M}v_0 := \partial_x^{-2}\partial_v F(v,M)v_0 = v_0 + \partial_x^{-2}\left[\Bigl(M\frac{1}{(v+1)(v+2)^2}-g\Bigr)v_0 - M \fint_{-\frac{\pi}{k_0^2}}^{-\frac{\pi}{k_0^2}} \frac{1}{(v+1)(v+2)^2}v_0 \de x\right].
    \end{equation*}
    Since \(\partial_x^{-2}\colon \Xcal \to \Xcal\) is  compact, we find that \(\tilde{L}_{v,M}\colon \Xcal \to \Xcal\) is of the form identity plus a compact operator. Hence, we can apply \cite[Thm. 2.7.6]{buffoni2003} to obtain that \(\tilde{L}_{v,M}\) is a Fredholm operator of index zero. Since \(\partial_x^2\colon \Xcal \to \Ycal\) is also Fredholm of index zero by invertibility, this shows that the composition \(\partial_v F(v,M) = \partial_x^2\tilde{L}_{v,M}\) is a Fredholm operator of index zero as well, see \cite[Thms. 16.5 and 16.12]{müller2007}.

    By standard elliptic regularity theory, solutions to the bifurcation equation \eqref{eq:bifurcation-problem} are smooth. Thus, the second statement follows by compact embedding.

    This proves both conditions \ref{it:global-1} and \ref{it:global-2} and  the statement follows from \cite[Thm. 9.1.1]{buffoni2003}.
\end{proof}

\subsection{Behaviour of the global bifurcation branch}\label{sec:boundary-behaviour}

In this step, we study the eventual behaviour of the global bifurcation curve obtained in Theorem \ref{thm:global-bifurcation}. First, we will rule out that the curve is a closed loop, hence alternative \ref{it:global-condition-3} cannot occur. Next, we will show that at least condition \ref{it:global-condition-2} does occur in the form of the solution approaching the film-rupture condition
\begin{equation*}
    \inf_{s\in \R} \min_{x} v(s) = -1.
\end{equation*}

\begin{proposition}\label{prop:no-closed-loop}
    Condition \ref{it:global-condition-3} in Proposition \ref{prop:global-bifurcation} does not occur.
\end{proposition}

In order to prove this proposition, we use \cite[Thm. 9.2.2]{buffoni2003} on global bifurcations in cones. We define the cone
\begin{equation}\label{eq:cone}
    \Kcal = \Bigl\{v \in \Xcal\,:\, v \text{ non-decreasing on }\bigl(-\tfrac{\pi}{k_0},0\bigr)\Bigr\}.
\end{equation}

\begin{proof}
    Following \cite[Thm. 9.2.2]{buffoni2003}, we need to check the following conditions.
    \begin{enumerate}[label=(\arabic*)]
        \item\label{it:nodal-1} \(\Kcal\) is a cone in \(\Xcal\);
        \item\label{it:nodal-2} there is \(\eps >0\) such that the local bifurcation branch satisfies \(v(s)\in \Kcal\) for \(0\leq s < \eps\);
        \item\label{it:nodal-3} if \(\xi \in \ker\partial_v F(0,M)\cap \Kcal\) for some \(M\in \R\), then \(\xi = \alpha\xi_0\) for some \(\alpha \geq 0\) and \(M=M^\ast(k_0)\).
        \item\label{it:nodal-4} \(\{(v,M)\in \Scal : v\neq 0\}\cap (\Kcal\times \R)\) is open in \(\Scal\).
    \end{enumerate}
    Notice that the last condition implies condition (d) of \cite[Thm. 9.2.2]{buffoni2003}.

    Indeed, \(\Kcal\) is a cone in \(\Xcal\) since it is closed and invariant under multiplication by non-negative scalars. For \ref{it:nodal-3} note that the Fourier symbol of \(L_{0,M}=\partial_v F(0,M)\) is given by
    \begin{equation*}
        \hat{L}_{0,M}(\ell) = -k_0^2\ell^2 + \frac{1}{4}M-g.
    \end{equation*}
    Since \(L_{0,M}\) operates diagonally on Fourier modes, for every \(M\) the kernel is given by the span of those Fourier modes \(\cos(k_0\ell \cdot)\) such that \(\hat{L}_{0,M}(\ell) = 0\). In particular, since \(\ell \geq 1\), each kernel is at most one-dimensional. The only Fourier mode to be in the kernel is given by the mode for \(\ell = 1\), that is \(\xi_0 = \cos(k_0\cdot)\). But \(\xi_0 \in \ker L_{0,M}\) if and only if \(M = M^\ast(k_0)\). Hence, condition \ref{it:nodal-3} is satisfied.

    To prove \ref{it:nodal-2} and \ref{it:nodal-4}, we follow roughly the ideas presented in \cite[Sec. 9.3]{buffoni2003}. We first prove \ref{it:nodal-4}. Let \((v,M) \in \Scal\cap (\Kcal\times \R)\), \((v,M)\neq 0\), i.e. we know that
    \begin{equation*}
        F(v,M) = \partial_x^2 v + f_M(v) = 0.
    \end{equation*}
    Assume by contradiction that \((v,M)\) is not in the interior of \(\Scal\cap (\Kcal\times \R)\). Then there exists a sequence \((v_n,M_n)\in \Scal\setminus (\Kcal \times \R)\) such that
    \begin{equation*}
        (v_n,M_n) \longrightarrow (v,M) \quad \text{in } \Xcal \times \R.
    \end{equation*}
    Since \(v_n\notin \Kcal\) and \(\Xcal\) embeds into \(C^1\), there further exists \(x_n \in \bigl(-\tfrac{\pi}{k_0},0\bigr)\) such that \(v_n'(x_n) < 0\). We have that \(v_n'(0) = 0 = v_n'\bigl(-\frac{\pi}{k_0}\bigr)\) using that \(v_n\) is even and periodic. Thus, we may  assume that \(x_n\in \bigl(-\tfrac{\pi}{k_0},0\bigr)\) is chosen such that \(v_n''(x_n) = 0\). Keep in mind that since \((v_n,M_n)\) is a solution to the bifurcation problem, \(v_n\) is smooth.

    Now, we may extract a convergent subsequence (not relabelled) and \(x^*\in \bigl[-\tfrac{\pi}{k_0},0\bigr]\) such that \(x_n \to x^*\). Since \(v_n\) converges to \(v\) in \(\Xcal\) and hence, by Sobolev embedding, in particular it holds \(v_n' \longrightarrow v'\) uniformly, we have
    \begin{equation*}
        0\leq v'(x^*) = \lim_{n\to \infty} v_n'(x_n) \leq 0,
    \end{equation*}
    where we use that \(v\in \Kcal\). Thus, \(v'(x^*) = 0\). To obtain that the second derivatives converge uniformly, notice that by standard elliptic regularity theory, see \cite[Thm. 6.3.2]{evans2010}, we obtain
    \begin{equation}\label{eq:elliptic-reg-estimate}
        \|v_n-v\|_{H^3_{\mathrm{per}}} \leq C\Bigl( \|v_n-v\|_{L^2_{\mathrm{per}}} + \|f_{M_n}(v_n)-f_M(v)\|_{H^1_{\mathrm{per}}} \Bigr).
    \end{equation}
    Since \(f \colon (-1,\infty)\to \R\) is smooth, the operator \(v \mapsto f_M(v)\) is a locally Lipschitz continuous Nemitski operator on \(\{v\in H^1_{\mathrm{per}}\,:\, v > c\}\) for every \(c>-1\), see \cite[Thm. 5.5.1]{runst1996}. Hence, we obtain the estimate
    \begin{equation*}
    \begin{split}
        \|f_{M_n}(v_n)-f_M(v)\|_{H^1_{\mathrm{per}}} & \leq \|f_{M_n}(v_n)-f_M(v_n)\|_{H^1_{\mathrm{per}}} + \|f_{M}(v_n)-f_M(v)\|_{H^1_{\mathrm{per}}} \\
        &\leq C\|v_n-v\|_{H^1_{\mathrm{per}}} + C|M_n-M|\sup_{n}\|v_n\|_{H^1_\mathrm{per}},
    \end{split}
    \end{equation*}
    since for $n$ sufficiently large,  $v_n$ satisfies 
    \begin{equation*}
        -1 < -\frac{1}{2} + \frac{1}{2}\min_{x\in \bigl[-\tfrac{\pi}{k_0},\tfrac{\pi}{k_0}\bigr]} v(x) \leq v_n(x) \leq 2 \max_{x\in \bigl[-\tfrac{\pi}{k_0},\tfrac{\pi}{k_0}\bigr]} v(x) < \infty.
    \end{equation*}
    We conclude that \(v_n\) converges to \(v\) in \(H^3_{\per}\) and by the Sobolev embedding, we obtain uniform convergence of \(v_n''\) to \(v\). Combining this with \(v_n''(x_n)=0\), it holds
    \begin{equation*}
        v''(x^*) = 0.
    \end{equation*}
    Differentiating the bifurcation equation for \(v\) with respect to \(x\) yields
    \begin{equation}\label{eq:elliptic-ODE-v'}
        \partial_x^3 v + f_M'(v)\partial_x v = 0.
    \end{equation}
    In particular, the non-local term in \(f_M(v)\) vanishes. Equation \eqref{eq:elliptic-ODE-v'} is an ordinary differential equation for \(\partial_x v\) with regular coefficients. Due to uniqueness of solutions to the initial-value problem with initial value \(v'(x^*)=0\) and $v''(x^\ast) = 0$, we obtain \(v\equiv 0\). This is a contradiction to the choice of \(v \in \Scal \setminus\{0\}\) and we obtain condition \ref{it:nodal-4}.

    Eventually, we prove condition \ref{it:nodal-2}. Recalling \eqref{eq:expansions}, it holds that
    \begin{equation*}
        v(s) = s \cos(k_0 x) + \tau(s)
    \end{equation*}
    with $\tau = \Ocal(\vert s \vert^2)$ in $\Xcal$. This is in $\Kcal$ if
    \begin{equation*}
        \partial_x v(s) = s k_0 \sin(k_0 x) + \partial_x \tau(s) \geq 0
    \end{equation*}
    for all $x \in \bigl(-\tfrac{\pi}{k_0},0\bigl)$. In order to be able to use uniform bounds in \(s\), we note that this is equivalent to
    \begin{equation*}
        \partial_x \frac{v(s)}{s} \geq 0
    \end{equation*}
    for all $x \in \bigl(-\tfrac{\pi}{k_0},0\bigl)$. To simplify notation, for \(s>0\) define \(\tilde{v}(s) = \tfrac{v(s)}{s}\) and $\tilde{\tau}(s) = \tfrac{\tau(s)}{s} = \Ocal(|s|)$ in $\Xcal$.
    
    We claim that it is enough to show that \(\partial_x^2 \tilde{v}(s) > \kappa >0\) uniformly for \(s\) small enough and in a neighborhood of \(x=-\frac{\pi}{k_0}\), while \(\partial_x^2 \tilde{v}(s) < - \kappa <0\) uniformly for \(s\) small enough in a neighborhood of \(x=0\). Since by the same regularity argument as in the proof of condition \ref{it:nodal-4}, we obtain uniform \(H^3\)-bounds for \(v(s)\), it is enough to show that \(\partial_x^2\tilde{v}(s)|_{x=-\tfrac{\pi}{k_0}}>\kappa >0\) and \(\partial_x^2\tilde{v}(s)|_{x=-0}<-\kappa <0\) uniformly for \(0\leq s < \eps\).

    Indeed, since \(\partial_x \tilde{v}(s)|_{x=0} = \partial_x \tilde{v}(s)|_{x=-\tfrac{\pi}{k_0}} = 0\) as \(\tilde{v}(s)\) is even and periodic, there exists \(\delta >0\) such that \(\partial_x \tilde{v}(s) \geq 0\) on \(\bigl[-\tfrac{\pi}{k_0},-\tfrac{\pi}{k_0}+\delta\bigr]\cup [-\delta,0]\) for all \(0\leq s < \eps\) and 
    \begin{equation*}
        \partial_x \tilde{v}(s)|_{x=-\tfrac{\pi}{k_0}+\delta} \geq \frac{1}{2}\delta\kappa>0, \quad \text{and} \quad \partial_x \tilde{v}(s)|_{x=-\delta} \geq \frac{1}{2}\delta\kappa>0.
    \end{equation*}
    Recalling that \(\tilde{v}(s) = \xi_0 + \tilde{\tau}(s)\) and using that \(\partial_x\xi_0 > c >0\) on \(\bigl[-\tfrac{\pi}{k_0}+\delta,-\delta\bigr]\) together with \(\tilde{\tau} = \Ocal(s)\) in $\Xcal\hookrightarrow C^1(\R)$, we may choose \(\eps>0\) even smaller to obtain
    \begin{equation*}
        \partial_x v(s) \geq 0 \text{ on } \bigl[-\tfrac{\pi}{k_0},0\bigr] \text{ for all } 0\leq s < \eps.
    \end{equation*}

    We are left to prove the bounds for the second derivative. Note that \(\tilde{v}(s) = \cos(k_0x) + \tilde{\tau}(s)\) and thus $\partial_x^2 \tilde{v}(s) = -k_0^2 \cos(k_0 x) + \partial_x^2 \tilde{\tau}(s)$. Since $\cos\bigl(-k_0 \tfrac{\pi}{k_0}\bigr) = -1$ and $\cos(0) = 1$ it is sufficient to show that $\partial_x^2 \tilde{\tau}(s) = \Ocal(|s|)$ in $C^0_{\mathrm{per}}$. To do this, we show that $\tau = \Ocal(|s|^2)$ in $H^3_\mathrm{per}$ by using elliptic regularity estimates. Note that the regularity estimates previously obtained for \(v\) are not sufficient since they would only imply that \(\tau=\Ocal(s)\) in \(H^3_{\mathrm{per}}\). To improve this estimates, we observe that $\tau$ satisfies
    \begin{equation}\label{eq:elliptic-tau}
        \partial_x^2 \tau - sk_0^2 \cos(k_0 x) + f_{M(s)}\bigl(s \cos(k_0 x) + \tau(s)\bigr) = 0.
    \end{equation}
    Next, we Taylor expand $v \mapsto f_M(v)$ at $v = 0$ in $\Xcal$ which, together with the expansion \eqref{eq:expansions} for \(M(s)\), gives
    \begin{equation*}
        f_{M(s)}(v) = 0 + \Bigl(- g + \frac{1}{4}(M^\ast(k_0) + \Ocal(|s|^2))\Bigl) v + \Ocal(\|v\|_{\Xcal}^2) = + k_0^2 v + \Ocal(|s|^2\|v\|_{\Xcal}) + \Ocal(\|v\|_{\Xcal}^2).
    \end{equation*}
    Inserting this expansion into \eqref{eq:elliptic-tau} and using $v(s) = \Ocal(|s|)$ yields
    \begin{equation*}
        \partial_x^2 \tau + g_s(v(s)) = 0
    \end{equation*}
    with $g_s(v(s)) = f_{M(s)}\bigl(s\cos(k_0x)+\tau(s)) \bigr)-sk_0^2\cos(k_0x) = \Ocal(|s|^2)$ in $\Xcal$. Using the elliptic regularity estimate \eqref{eq:elliptic-reg-estimate}, c.f. the proof of \ref{it:nodal-4}, we find that
    \begin{equation*}
        \|\tau(s)\|_{H^3_{\mathrm{per}}} = \Ocal(|s|^2),
    \end{equation*}
    which completes the proof of both the claim and the proposition.
\end{proof}

\begin{remark}\label{rem:smallness-regularity}
    In fact, we remark that the remainder term $\tau(s)$ is of order $s^2$ in $H^m_\mathrm{per}$ for all $m \geq 0$. This can be found by iteratively applying the elliptic regularity estimate
    \begin{equation*}
        \|\tau(s)\|_{H^m_\mathrm{per}} \leq C (\|\tau(s)\|_{L^2_\mathrm{per}} + \|g_s(v(s))\|_{H^{m-2}_\mathrm{per}}),
    \end{equation*}
    see \cite[Thm. 6.3.2]{evans2010}. Here, we use that $g_s(v(s))$ is analytic and hence we obtain the estimate
    \begin{equation*}
        \|g_s(v(s))\|_{H^{m-2}_\mathrm{per}} \leq  C\bigl(\|\tau(s)\|_{H^{m-2}_\mathrm{per}} + |M(s)-M^*(k_0)| \|v(s)\|_{H^{m-2}_\mathrm{per}} + \|v(s)\|_{H^{m-2}_\mathrm{per}}^2\bigr) = \Ocal(|s|^2).
    \end{equation*}
\end{remark}

\begin{remark}\label{rem:nodal-property}
    We point out that the proof of \ref{it:nodal-4} shows that if a periodic solution $v$ to \eqref{eq:steady-state-equation} is non-decreasing in $(-\tfrac{\pi}{k_0},0)$, then it is strictly increasing in $(-\tfrac{\pi}{k_0},0)$. Since all solutions $v(s)$ on the global bifuration curve lie in the cone $\Kcal$, this means that the $v(s)$ have their only minima on $[-\tfrac{\pi}{k_0},\tfrac{\pi}{k_0}]$ at $x = \pm \tfrac{\pi}{k_0}$ and their only maximum at $x = 0$. These minima and maxima are invariant along the bifurcation curve and thus the solutions form a nodal pattern. This is also referred to as nodal property, see e.g.~\cite[Theorem 4.9]{ehrnström2019}.
\end{remark}

Now we have ruled out the possibility of the global bifurcation branch to form a closed loop. We show next that the curve approaches the boundary of the domain \(\Ucal\) in the sense that \(\min_{x}v(s)\) approaches \(-1\) at which point film rupture occurs.

\begin{proposition}\label{prop:film-rupture}
    Let \(\{(v(s),M(s))\in \Ucal \times \R : s \in \R\}\) the global bifurcation branch obtained in Proposition \ref{prop:global-bifurcation}. Then, as \(s\to \infty\), \(v(s)\) approaches the boundary of \(\Ucal\) in the sense that
    \begin{equation*}
        \inf_{s\in \R} \min_{x\in \bigl[-\tfrac{\pi}{k_0},\tfrac{\pi}{k_0}\bigr]}v(s) = -1.
    \end{equation*}
\end{proposition}

\begin{proof}
    We argue by contradiction and assume that
    \begin{equation}\label{eq:lower-bound-min-v}
        \min_{x\in \bigl[-\tfrac{\pi}{k_0},\tfrac{\pi}{k_0}\bigr]} v(s) \geq c > -1 \quad \text{for all } s\in \R.
    \end{equation}
    In this case, the alternatives of Theorem \eqref{thm:global-bifurcation} imply that \(\|v(s)\|_{\Xcal}\to+\infty\) as \(s\to \infty\) or \(|M(s)|\to \infty\) as \(s\to \infty\). We will rule out both possibilities, leading to a contradiction. The remainder of the proof is structured in three steps. In \hyperlink{film-rupture-step1}{step 1}, we establish pointwise upper bounds for the maximum and the minimum of $v(s)$ using the Hamiltonian structure of the equation. In \hyperlink{film-rupture-step2}{step 2}, we then establish lower and upper bounds on $M(s)$ using assumption \eqref{eq:lower-bound-min-v}. Finally, in \hyperlink{film-rupture-step3}{step 3}, we combine the results of steps \hyperlink{film-rupture-step1}{1} and \hyperlink{film-rupture-step2}{2} to obtain uniform bounds on the $H^2_\mathrm{per}$-norm of $v(s)$.
    
    \noindent\textbf{Step 1: pointwise bounds for $v(s)$ using the Hamiltonian system.}\linkdest{film-rupture-step1}{} We first show that we obtain pointwise upper bounds for \(v(s)\) using the fixed points of the Hamiltonian system \eqref{eq:Hamiltonian}. Recall that any periodic solution is a periodic orbit of the Hamiltonian system with Hamiltonian
    \begin{equation*}
        \Hcal(v,w) = \frac{1}{2}w^2 - \frac{g}{2} v^2 + M(1+v)\log\Bigl(\frac{1+v}{2+v}\Bigr) - MK(v)v
    \end{equation*}
    with $K(v)$ given by
    \begin{equation*}
        K(v) = \fint_{-\frac{\pi}{k_0}}^{\frac{\pi}{k_0}} \frac{1}{2+v} + \log\Bigl(\frac{1+v}{2+v}\Bigr) \de x.
    \end{equation*}
    In order to apply the results of Lemma \ref{lem:orbits-Hamiltonian-system}, we now show that $K(v)$ satisfies the estimate
    \begin{equation*}
        K(v) \leq K(0).
    \end{equation*}
    This follows from the fact that $v \mapsto \tfrac{1}{2+v} + \log\bigl(\tfrac{1+v}{2+v}\bigr) = \omega(v)$ is strictly concave and thus
    \begin{equation*}
        K(v) = \fint_{-\frac{\pi}{k_0}}^{\frac{\pi}{k_0}} \omega(v) \de x \leq \omega\left(\fint_{-\frac{\pi}{k_0}}^{\frac{\pi}{k_0}} v \de x\right)= \omega(0) = K(0),
    \end{equation*}
    by Jensen's inequality and the fact that the mean value of $v$ vanishes. Notice that we have equality if and only if \(v=0\). Since Proposition \ref{prop:no-closed-loop} rules out the possibility of the bifurcation curve to be a closed loop, we have that $v(s) \neq 0$ for $s \neq 0$. Therefore, we can always assume that $K < K(0)$. Then, Lemma \ref{lem:orbits-Hamiltonian-system} guarantees that there are two distinct fixed points $(v_l,0)$ and $(v_u,0)$ of the Hamiltonian system $\Hcal$ with $-1 < v_l < 0 < v_u$. In particular, since any periodic orbit loops around $(v_l,0)$, we have the estimate
    \begin{equation*}\label{eq:bound-min-v}
        c \leq \min_{x\in \bigl[-\tfrac{\pi}{k_0},\tfrac{\pi}{k_0}\bigr]} v(s) \leq v_l.
    \end{equation*}
    Additionally, we also find that
    \begin{equation}\label{eq:bound-max-v}
        \max_{x\in \bigl[-\tfrac{\pi}{k_0},\tfrac{\pi}{k_0}\bigr]} v(s) \leq v_u
    \end{equation}
    since the neighborhood $U$ filled with periodic orbits is bounded either by the level set of $v = -1$ or $v = v_u$ with smaller energy.
    By assumption \eqref{eq:lower-bound-min-v} we also have that $K(v) \geq K_l(c) > -\infty$.

    \noindent \textbf{Step 2: bounds for $\mathbf{M(s)}$.}\linkdest{film-rupture-step2}{} We now show that $M(s) \in (M_l,M_u)$ with constants $0< M_l < M_u < \infty$, provided we assume \eqref{eq:lower-bound-min-v}. Since $v_l$ satisfies \eqref{eq:fixed-point} and recalling Figure \ref{fig:fixed-points}, we obtain that $M \rightarrow 0$ implies that $v_l \rightarrow -1$ for any $K \in [K_l(c),K(0))$. Since we assume \eqref{eq:lower-bound-min-v}, this implies that $M(s) > M_l > 0$.

    To obtain the upper bound for \(M(s)\), we show that the maximal period of the periodic orbits of the Hamiltonian system $\Hcal$ tends to $0$ as $M \rightarrow \infty$. Since  $v(s)$ has the fixed period $\tfrac{2\pi}{k_0}$ this leads to a contradiction and shows that $M < M_u < \infty$. 
    
    To show that the maximal period vanishes, we introduce the Hamiltonian
    \begin{equation*}
        \tilde{\Hcal}(v,w) = \frac{1}{2} w^2 - \frac{g}{2M} v^2 + (1+v)\log\Bigl(\frac{1+v}{2+v}\Bigr) - K v = \frac{1}{2} w^2 + \Gcal(v)/M,
    \end{equation*}
    where we recall that $\Hcal(v,w) = \frac{1}{2}w^2 + \Gcal(v)$, see \eqref{eq:Gcal} for the definition of $\Gcal$. This Hamiltonian $\tilde{\Hcal}$ is rescaled in such a way to make the relevant dynamics of $\Hcal$ visible when $M$ is large. We now show that periodic orbits of $\tilde{\Hcal}$ have maximal period by treating $\tfrac{g}{M} v^2$ as a small perturbation for $M$ sufficiently large. The period can be calculated by
    \begin{equation*}
        T_{\tilde{\Hcal}}(E) = \sqrt{2} \int_{q_0}^{q_1} \frac{1}{\sqrt{E - \tfrac{\Gcal(v)}{M}}} \de v = \sqrt{M} \sqrt{2} \int_{q_0}^{q_1} \frac{1}{\sqrt{M E - \Gcal(v)}} \de v = \sqrt{M} T_{\Hcal}(ME),
    \end{equation*}
    where $q_0,q_1$ are the lower and upper root of $q \mapsto E - \tilde{\Hcal}(q,0)$, with $E$ being an energy level at which periodic solutions exists, see e.g. \cite{chicone1987}. Observe that \(q_0,q_1\) are also the lower and upper root of \(q\mapsto ME-\Hcal(q,0)\). Notice that
    \begin{equation*}
        \tilde{\Hcal}(v,0) = \frac{\Gcal(v)}{M} = \frac{\Hcal(v,0)}{M}
    \end{equation*}
    and thus, if $I(\tilde{\Hcal})$ is the set of all energy levels corresponding to periodic orbits in $\tilde{\Hcal}$, then $I(\Hcal) = M I(\tilde{\Hcal})$. Additionally, following the proof of Lemma \ref{lem:orbits-Hamiltonian-system} it holds that $I(\tilde{\Hcal}) = (\tilde{\Hcal}(v_l,0), \min(\tilde{\Hcal}(-1,0),\tilde{\Hcal}(v_u,0)))$ is an interval.
    
    We show next that \(T_{\tilde{\Hcal}}(E)\) is bounded from above for all \(E\in I(\tilde{\Hcal})\), all \(M\) large enough, and all \(K \in (K_l(c),K(0))\). We observe that $v \mapsto \tfrac{\Gcal(v)}{M}$ and all its derivatives converge uniformly on every compact interval to 
    \begin{equation*}
        \Gcal_\infty(v) = (1+v) \log\bigl(\tfrac{1+v}{2+v}\bigr) - Kv
    \end{equation*}
    as $M \rightarrow \infty$. Additionally, we conclude from \eqref{eq:fixed-point} and Figure \ref{fig:fixed-points} that $v_u \rightarrow \infty$ as $M \rightarrow \infty$. Hence, for $M$ sufficiently large, it holds that $\Gcal_\infty(-1) < \Gcal_\infty(v_u/2)$.

    \begin{figure}[H]
    \centering
    \includegraphics[width=0.5\textwidth]{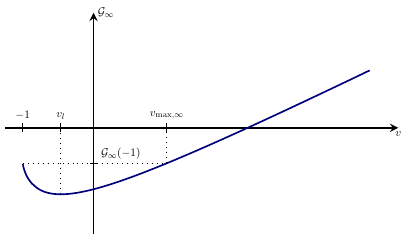}
    \caption{Plot of \(\Gcal_{\infty}\) with $K = -0.4 < K(0) \approx -0.193$.}
    \label{fig:energy-function}
\end{figure}
    
    Next, we fix such a sufficiently large $M_u < \infty$ and define $V = [-1,2v_{\mathrm{max},\infty}]$ with $\Gcal_\infty(-1) = \Gcal_\infty(v_{\mathrm{max},\infty})$ and \(v_{\mathrm{max},\infty}>-1\). The existence and uniqueness of such a $v_{\mathrm{max},\infty}$ follows since $\Gcal_\infty$ is uniformly convex, see Figure \ref{fig:energy-function}. In fact, $v_{\mathrm{max},\infty}$ is explicitly given by
    \begin{equation*}
        v_{\mathrm{max},\infty} = -\frac{2 e^K-1}{e^K-1}.
    \end{equation*}
    In particular, since $v \mapsto \tfrac{\Gcal(v)}{M}$ converges uniformly to $\Gcal_\infty$ on $V$, we can enforce that $\tfrac{\Gcal(-1)}{M} < \tfrac{\Gcal(2v_{\mathrm{max},\infty})}{M}$ and \(\tfrac{\Gcal}{M}\) is uniformly convex for \(v\in V\) uniformly for \(M>M_u\) choosing $M_u$ potentially even larger.

    Then the upper fixed point \(v_u\) satisfies \(\tfrac{\Gcal(v_u)}{M} > \tfrac{\Gcal(2 v_{\mathrm{max},\infty})}{M} > \frac{\Gcal(-1)}{M}\) since \(v_u\) is a maximum of $\Gcal$ according to Lemma \ref{lem:orbits-Hamiltonian-system} and thus of \(\tfrac{\Gcal}{M}\). This implies that there is no homoclinic orbit and the level of \(\tilde{\Hcal}(-1,0)\) encloses the set of periodic orbits.
    
    Since the level set of \((-1,0)\) is contained in \(V\times \R\), we obtain that all periodic orbits of $\tilde{\Hcal}$, which loop around $(v_l,0) \in V \times \{0\}$, are contained in the set $V \times \R$. In particular, for all $E \in I(\tilde{\Hcal})$ we find that the corresponding $q_0$ and $q_1$ are bounded away from $v_u$. Additionally, if $E \geq \tilde{\Hcal}(v_l,0) + \delta$ for $\delta$ small, the corresponding $q_0$ and $q_1$ are also bounded away from $v_l$. This yields that $T_{\tilde{\Hcal}}(E)$ is uniformly bounded for $\tilde{\Hcal}(v_l,0) + \delta \leq E \leq \tilde{\Hcal}(-1,0)$, since
    \begin{equation*}
        E - \frac{\Gcal(v)}{M} = \frac{\Gcal'(v)}{M}(v-q_1) + \Ocal(|v-q_1|^2)
    \end{equation*}
    by Taylor expansion around $v = q_1$ (similar for $q_0$) and thus $\Bigl(\sqrt{E - \tfrac{\Gcal(v)}{M}}\Bigr)^{-1}$ behaves like $\tfrac{1}{\sqrt{v}}$ for $v$ close to $q_0$ or $q_1$ and thus is integrable.

    To control $T_{\tilde{\Hcal}}(E)$ for $\tilde{\Hcal}(v_l,0) < E \leq \tilde{\Hcal}(v_l,0) + \delta$, we use that
    \begin{equation*}
        J D^2 \tilde{H}(v_l,0) = \begin{pmatrix}
            0 & 1 \\
            - \tfrac{\Gcal''(v_l)}{M} & 0
        \end{pmatrix} = \begin{pmatrix}
            0 & 1 \\
            \tfrac{g}{M} - \tfrac{1}{(v_l+1)(v_+2)^2} & 0
        \end{pmatrix}.
    \end{equation*}
    Therefore, the eigenvalues of $J D^2 \tilde{\Hcal}$ are given by
    \begin{equation*}
        \lambda_{\pm} = \pm\sqrt{\frac{1}{(v_l+1)(v_l+2)^2} - \frac{g}{M}}.
    \end{equation*}
    Using the Lyapunov subcentre theorem \cite[Thm. 4.1.8]{schneider2017}, we know that the periods of the periodic orbits as \(E\to \tilde{\Hcal}(v_l,0)\) are close to \(\frac{2\pi}{\lambda_+}\). Since \(\lambda_+\) is bounded away from zero as \(M>M_u\), the period is also uniformly (in \(M\)) bounded for \(E\in[\tilde{\Hcal}(v_l,0),\tilde{\Hcal}(v_l,0)+\delta]\). We conclude, that the period remains bounded.

    \noindent\textbf{Step 3: $\mathbf{v(s)}$ is bounded in \(\Xcal\).}\linkdest{film-rupture-step3}{} We show that \(\|v(s)\|_{L^{\infty}}\) remains uniformly bounded for all \(s\geq 0\). 
    Indeed, once we have shown this, we have that \(f_{M(s)}(v(s))\) is uniformly bounded in \(L^2_{\mathrm{per}}\), where we also make use of the fact that \(M(s) \in [M_l,M_u]\) is uniformly bounded. Now using standard elliptic regularity theory, see e.g.~\cite[Thm. 6.3.1]{evans2010}, we may conclude that \(v(s)\) is uniformly bounded in \(\Xcal\). 
    Thus, it remains to show that \(\|v(s)\|_{L^{\infty}}\) remains uniformly bounded. By \eqref{eq:bound-max-v} it is sufficient to bound the upper fixed point of the Hamiltonian system $\Hcal$. Recalling that $v_u$ solves \eqref{eq:fixed-point} and Figure \ref{fig:fixed-points} and using that $M \in (M_l,M_u)$ and $K \in (K_l(c),K(0))$ we find that $v_u$ enjoys the estimate
    \begin{equation}\label{eq:uniform-bound-L-infty}
        v_u \leq - K_l(c)\frac{M_u}{g} < \infty.
    \end{equation}
    This yields the statement.

    We conclude that both remaining alternatives \ref{it:global-condition-1} and \ref{it:global-condition-2} of Proposition \ref{prop:global-bifurcation} are ruled out and therefore obtain the contradiction that \eqref{eq:lower-bound-min-v} cannot hold. This implies that the global bifurcation curve approaches the boundary of \(U\) at which point film rupture occurs. 
\end{proof}

\section{Limit points of the bifurcation branch are weak stationary film-rupture solutions}\label{sec:film-rupture}

We study the limit points of the global bifurcation branch obtained in Proposition \ref{prop:global-bifurcation} to show that in the limit \(s\to \infty\) we obtain a weak stationary periodic solution to the thermocapillary thin-film equation \eqref{eq:thin-film-equation} that exhibits film rupture.

\subsection{Uniform bounds for the bifurcation branch}\label{sec:uniform-bounds}

In order to identify limit points of the global bifurcation branch, we begin by establishing uniform bounds both for \(M(s)\) and \(v(s)\). In the following lemma, we first provide estimates for \(M(s)\) and \(K\).

\begin{lemma}\label{lem:uniform-bound-M}
    Let \(\{(v(s),M(s)) : s \in \R\} \subset \Ucal \times \R\) be the global bifurcation branch obtained in Proposition \ref{prop:global-bifurcation} and
    \begin{equation*}
          K(s) = K(v(s)) = \fint_{-\frac{\pi}{k_0}}^{\frac{\pi}{k_0}} \frac{1}{2+v(s)} + \log\Bigl(\frac{1+v(s)}{2+v(s)}\Bigr) \de x.
    \end{equation*}
    Then there are constants \(0<M_l < M_u < \infty\) and \(-\infty < K_l < K_u < K(0) = \frac{1}{2} + \log\bigl(\tfrac{1}{2}\bigr)< 0\) such that
    \begin{equation*}
        M_l \leq M(s) \leq M_u\quad \text{and}\quad K_l \leq K(s) \leq K(0)
    \end{equation*}
    holds for all \(s\in \R\) and $K(s) < K_u < K(0)$ for $s > 1$.
\end{lemma}

\begin{proof}
    We recall from Lemma \ref{lem:orbits-Hamiltonian-system} that the boundary of the set of all periodic orbits around $(v_l,0)$ is either given by the level set of $\Hcal(-1,0)$ or the level set of $\Hcal(v_u,0)$ if $\Hcal(v_u,0) \leq \Hcal(-1,0)$. In particular, this implies that any periodic orbit around $(v_l,0)$ satisfies
    \begin{equation*}
        v \leq v_\mathrm{max},
    \end{equation*}
    where $v_\mathrm{max} := \min\{v > -1 \,:\, \Hcal(v,0) = \Hcal(-1,0)\}$.
    In particular, we note that if $v_\mathrm{max} < 0$ for some $K,M$, then there are no periodic solutions with vanishing mass.

    We split the proof into three steps. First, we show that $M(s)$ is bounded from above independently of the behaviour of $K(s)$. Second, we bound $M(s)$ from below again independently of the behaviour of $K(s)$. Finally, using that $M(s) \in (M_l,M_u)$ we find that $K(s)$ is bounded from below.

    \noindent\textbf{Step 1: upper bound for $\mathbf{M(s)}$.} To obtain an upper bound on $M(s)$ we recall from \hyperlink{film-rupture-step2}{step 2} of the proof of Proposition \ref{prop:film-rupture} that if $K \geq K_l > -\infty$ then $M \leq M_u < \infty$. Now, we show that $M(s) \rightarrow +\infty$ implies that $K(s) \geq K_l > -\infty$, which leads to a contradiction and thus $M(s) \leq M_u < \infty$.

    We follow the proof of Proposition \ref{prop:film-rupture} and recall that
    \begin{equation*}
        \Gcal_\infty(v) = (1+v) \log\Bigl(\frac{1+v}{2+v}\Bigr) - Kv.
    \end{equation*}
    Next, we note that
    \begin{equation*}
        \Bigl\vert\frac{\Gcal(v)}{M} - \Gcal_\infty(v)\Bigr\vert = \frac{g}{2M}v^2,
    \end{equation*}
    which is independent of $K$. Hence, if $M(s) \rightarrow \infty$ then $\tfrac{\Gcal}{M}$ converges to $\Gcal_\infty$ uniformly in $K$ and on compact intervals of $v$. In particular, if $|K(s)|$ is sufficiently large, then $v_{\mathrm{max},\infty}$ is negative (recall that \(v_{\mathrm{max},\infty}\) is the unique solution \(v > -1\) to \(\Gcal_{\infty}(v) = \Gcal_{\infty}(-1)\)). Thus, $v_\mathrm{max}$ is also negative for $M$ sufficiently large, which is a contradiction to the fact that the periodic solutions have vanishing mass. We conclude that \(M(s) \to +\infty\) implies that \(K(s)\) is bounded from below. This rules out, by contradiction, that \(M(s) \to +\infty\).

    \noindent\textbf{Step 2: lower bound for $\mathbf{M(s)}$.} Again, we argue by contradiction and assume that $M(s) \rightarrow 0$. We estimate
    \begin{equation*}
        \Gcal(v) - \Gcal(-1) = \frac{g}{2}(1-v^2) + M(1+v)\log\Bigl(\frac{1+v}{2+v}\Bigr) - MKv + MK \geq \frac{g}{2}(1-v^2) + M(1+v)\log\Bigl(\frac{1+v}{2+v}\Bigr),
    \end{equation*}
    using that $M \geq 0$ and $K < 0$. Notice that $(1+v) \log\bigl(\tfrac{1+v}{2+v}\bigr) \in (-1,0)$ and thus the lower bound converges to $\tfrac{g}{2}(1-v^2)$ uniformly in $v$. In particular, $\tfrac{g}{2}(1-v^2) > 0$ for $v \in (-1,0)$. This then implies that $v_\mathrm{max} < 0$ for $M$ sufficiently small and therefore, again by contradiction, we find that $M(s) \geq M_l > 0$.

    \noindent\textbf{Step 3: lower bound for $\mathbf{K(s)}$.} We use that $M \in [M_l,M_u]$. Additionally, using that $K < 0$, we consider
    \begin{equation*}
        \hat{\Gcal}(v) := \frac{\Gcal(v)}{-K} = \frac{g}{2K} v^2 - \frac{M}{K}(1+v)\log\Bigl(\frac{1+v}{2+v}\Bigr) + Mv.
    \end{equation*}
    Assuming that $K(s) \rightarrow -\infty$, the first two terms vanish uniformly on compact intervals in $v$, in particular for $v \in (-1,1)$ since $M \in [M_l,M_u]$. Since the last remaining term is increasing in $v$, we find that $v_\mathrm{max}$ is negative for $|K(s)|$ sufficiently large. Again, this contradicts the existence of periodic solutions with vanishing mass. Hence, there exists a lower bound \(-\infty <K_l \leq K(s)\). Finally, \(K(s) = K(0)\) implies that \(v(s) \equiv 0\) is a constant solution since $K$ is defined as an integral over a strictly concave function. By Proposition \eqref{prop:no-closed-loop} this only occurs for \(s=0\) and we conclude that \(K(s) \leq K_u < K(0)\) for \(s \geq 1\). This concludes the proof.
\end{proof}

Using the bounds for \(M(s)\) and \(K(s)\) obtained in Lemma \ref{lem:uniform-bound-M}, we obtain uniform bounds for the set of bifurcating solutions \(\{v(s)\,:\, s\in \R\}\).

\begin{proposition}\label{prop:uniform-bound-H2}
    Let \(\{(v(s),M(s)) : s \in \R\} \subset \Ucal \times \R\) be the global bifurcation branch obtained in Proposition \ref{prop:global-bifurcation}. Then \(\{v(s) \,:\, s\in \R\}\) is uniformly bounded in \(\Xcal\). In particular, there exists a constant \(C>0\) such that
    \begin{equation*}\label{eq:uniform-bound-H2}
        \|v(s)\|_{H^2_{\mathrm{per}}} \leq C.
    \end{equation*}
    Moreover, \(\{v(s) \,:\, s\in \R\}\) is uniformly bounded in \(W^{2,p}_\mathrm{per}\) for every \(p\in [1,\infty)\), but blows up in \(W^{2,\infty}_{\mathrm{per}}\).
\end{proposition}

\begin{remark}
    \begin{enumerate}
        \item Proposition \ref{prop:uniform-bound-H2} and Lemma \ref{lem:uniform-bound-M} imply in particular that condition \ref{it:global-condition-1} in Proposition \ref{prop:global-bifurcation} does not occur. Hence, only condition \ref{it:global-condition-2} occurs.
        \item Note that the $C^2$-norm of $v(s)$ is not uniformly bounded since the right-hand side of \eqref{eq:shifted-steady-state-equation} becomes unbounded as $v$ approaches $-1$. In particular, we cannot obtain uniform bounds in $H^3_\mathrm{per}$.
    \end{enumerate}
\end{remark}

\begin{proof}
    We first prove a uniform bound in \(W^{1,\infty}(\R)\) using the Hamiltonian system \eqref{eq:Hamiltonian}. Then, we use the differential equation \eqref{eq:shifted-steady-state-equation} to deduce a uniform bound in \(H^2_\mathrm{per}\).

    \noindent \textbf{Step 1: uniform bound in \(\mathbf{W^{1,\infty}(\R)}\).}\linkdest{uniform-bound-step1}{} As in the proof of Proposition \ref{prop:film-rupture} we obtain an upper bound for $v$ by using that
    \begin{equation*}
         v_u \leq -K_l \frac{M_u}{g} < \infty,
    \end{equation*}
    see \eqref{eq:uniform-bound-L-infty}. Here $K_l$ and $M_u$ are the lower bound on $K$ and upper bound on $M$ provided in the previous Lemma \ref{lem:uniform-bound-M}. This implies that $v(s)$ is uniformly bounded in $L^\infty(\R)$.

    To bound $w^2 = (\partial_x v)^2$, we use that the periodic orbits lie on the level sets of the Hamiltonian $\Hcal$. Hence, $w^2$ satisfies
    \begin{equation*}
        \frac{1}{2}w^2 = \frac{g}{2} v^2 - M(1+v)\log\Bigl(\frac{1+v}{2+v}\Bigr) + MKv + \Hcal(v,w).
    \end{equation*}
    Since $v$ is bounded in $L^\infty(\R)$, $M$ and $K$ are bounded from above and below by Lemma \ref{lem:uniform-bound-M}, and the fact that $(1+v)\log\bigl(\tfrac{1+v}{2+v}\bigr) \in (-1,0)$, we obtain  that $w^2$ is uniformly bounded in \(L^{\infty}(\R)\) if $\Hcal(v,w)$ is bounded. In view of \(v(s)\) being a family of periodic solutions, it is sufficient to show that $\Hcal(v,0)$ is bounded on $\bigl(-1,-K_l\tfrac{M_u}{g}\bigr)$. Since $v \mapsto \Hcal(v,0)$ can be continuously extended in $v = -1$, the uniform bound follows from continuity. This implies that \(w=\partial_x v\) is uniformly bounded in \(L^{\infty}(\R)\).

    \noindent \textbf{Step 2: uniform bound in \(\mathbf{H^2_{\mathrm{per}}}\).}\linkdest{uniform-bound-step2}{} To obtain a uniform bound for the \(L^2_\mathrm{per}\)-norm, we use that \(v(s)\) is a solution to the differential equation
    \begin{equation*}
        \partial_x^2 v(s) = gv(s) - M\Bigl(\frac{1}{2+v(s)} + \log\Bigl(\frac{1+v(s)}{2+v(s)}\Bigr)\Bigr) + M(s)K(s).
    \end{equation*}
    Hence, it is sufficient to bound the \(L^2_{\mathrm{per}}\)-norm of the right-hand side. Since \(v(s)\) is uniformly bounded in \(L^{\infty}(\R)\) by \hyperlink{uniform-bound-step1}{step 1} with \(v(s) > -1\) for all \(s\in \R\) and furthermore, by Lemma \ref{lem:uniform-bound-M} \(|M(s)|\) and \(|K(s)|\) are uniformly bounded, it suffices to obtain a uniform bound for
    \begin{equation*}
        \int_{-\frac{\pi}{k_0}}^{\frac{\pi}{k_0}} \Bigl|\log\Bigl(\frac{1+v(s)}{2+v(s)}\Bigr)\Bigr|^2 \de x \leq 2\int_{-\frac{\pi}{k_0}}^{\frac{\pi}{k_0}} |\log(1+v(s))|^2 \de x + C.
    \end{equation*}
    By Proposition \ref{prop:no-closed-loop} and Remark \ref{rem:nodal-property} the minima of \(v(s)\) on the interval \(\bigl[-\tfrac{\pi}{k_0},\tfrac{\pi}{k_0}\bigr]\) lie exactly at the boundary points \(\pm\tfrac{\pi}{k_0}\). Without loss of generality it is enough to study the integral close to \(x_{\mathrm{min}}=\tfrac{\pi}{k_0}\), since
    \begin{equation*}
        \int_{-\frac{\pi}{k_0}}^{\frac{\pi}{k_0}} |\log(1+v(s))|^2 \de x = \int_{0}^{\frac{2\pi}{k_0}} |\log(1+v(s))|^2 \de x
    \end{equation*}
    by periodicity of \(v(s)\). We will show that there exists \(0 < \eps < \tfrac{1}{2}\), a radius \(r>0\), and \(N_0>0\) such that
    \begin{equation*}
        v(s,x) + 1 \geq v(s,x)-v(s,x_{\mathrm{min}}) \geq |x-x_\mathrm{min}|^2 \quad \text{and} \quad v(s,x) + 1 \leq \frac{3}{2}\eps
    \end{equation*}
    for every \(|x-x_\mathrm{min}|<r\) and \(s \geq N_0\) and furthermore
    \begin{equation*}
        v(s,x) + 1 \geq \frac{\eps}{2} 
    \end{equation*}
    for every \(|x-x_\mathrm{min}|\geq r\) and \(s \geq N_0\). Indeed, this suffices for the uniform bound since
    \begin{equation*}
    \begin{split}
        \int_{0}^{\frac{2\pi}{k_0}} |\log(1+v(s))|^2 \de x & = \int_{B_r(x_\mathrm{min})} |\log(1+v(s))|^2 \de x + \int_{\bigl[0,\tfrac{2\pi}{k_0}\bigr]\setminus B_r(x_\mathrm{min})} |\log(1+v(s))|^2 \de x \\
        & \leq \int_{B_r(x_\mathrm{min})} |\log(|x-x_\mathrm{\min}|^2)|^2 \de x + C \\
        & \leq C,
    \end{split}
    \end{equation*}
    where $B_r(x_\mathrm{min})$ denotes the open ball with radius $r$ around $x_\mathrm{min}$.
    To obtain the pointwise bounds for \(v(s)\), notice that by \hyperlink{uniform-bound-step1}{step 1} there exists a periodic function \(v_{\infty}\in W^{1,\infty}(\R)\) and a sequence \(s_n \to \infty\) such that \(v(s_n) \overset{\ast}{\rightharpoonup} v_{\infty}\) as \(n\to \infty\) converges in the weak-star sense. In particular, \(v_{\infty}\) is continuous and \(v(s_n)\) converges to \(v_{\infty}\) uniformly. Furthermore, using Proposition \ref{prop:film-rupture} we know that \(v_{\infty}(x_{\min}) = -1\). Since \(|K(s)|\) is uniformly bounded and $v(s_n)$ converges uniformly to $v_{\infty}$, we obtain that \(v_{\infty}(x) > -1\) for all \(x\) in a neighborhood of \(x_{\mathrm{min}}\). Since \(v_{\infty}\in \Kcal\), c.f.~\eqref{eq:cone} for the definition of \(\Kcal\), we conclude that \(x_{\mathrm{min}}\) is the unique minimum of \(v_{\infty}\) in \(\bigl[0,\tfrac{2\pi}{k_0}\bigr]\).

    Now, choose \(0 < \eps <\frac{1}{2}\) and \(r = r(\eps) >0\) maximal so that
    \begin{equation}\label{eq:eps-bound-1}
        v_{\infty}(x) + 1 < \eps \quad \text{for all }x\in B_r(x_\mathrm{min}).
    \end{equation}
    Choose \(N_0 = N_0(\eps)\) large enough so that
    \begin{equation}\label{eq:eps-bound-2}
        \|v(s_n) - v_{\infty}\|_{L^{\infty}} < \frac{\eps}{2} \quad \text{for all } n \geq N_0.
    \end{equation}
    Using equation \eqref{eq:shifted-steady-state-equation}, we obtain the lower bound for the second derivative given by
    \begin{equation*}
    \begin{split}
        \partial_x^2 v(s_n) & = gv(s_n) - M \log(1+v(s_n)) + M\log(2+v(s_n)) - M \frac{1}{2+v(s_n)} + M(s_n)K(s_n)\\
        & \geq M |\log(1+v(s_n))| - M_u(1-K_l) - g.
    \end{split}
    \end{equation*}
    Since \eqref{eq:eps-bound-1} and \eqref{eq:eps-bound-2} imply that 
    \begin{equation*}
        v(s_n,x)+1 < \frac{3}{2} \eps \quad  \text{for all }x\in B_r(x_\mathrm{min}) \text{ and all } n \geq N_0,
    \end{equation*}
    we can choose \(\eps >0\) small enough such that 
    \begin{equation*}
        M |\log(1+v(s_n))| - M_u(1-K_l) - g > 2\quad  \text{for all }x\in B_r(x_\mathrm{min}) \text{ and all } n\geq  N_0.
    \end{equation*}
    By Taylor's formula with remainder, we obtain
    \begin{equation*}
        v(s_n,x) = v(s_n,x_\mathrm{min}) + \frac{1}{2}\partial_x^2 v(\xi)|x-x_\mathrm{min}|^2 \quad \text{for all }x\in B_r(x_\mathrm{min})
    \end{equation*}
    for some \(\xi=\xi(x) \in B_r(x_\mathrm{min})\). In particular, we obtain the lower bound
    \begin{equation*}
        v(s_n,x) - v(s_n,x_\mathrm{min}) \geq |x-x_\mathrm{min}|^2\quad  \text{for all }x\in B_r(x_\mathrm{min}) \text{ and all } n \geq N_0.
    \end{equation*}
    Furthermore, since \(v_{\infty} \in \Kcal\) it is monotonically increasing for \(\tfrac{\pi}{k_0}<x<\tfrac{2\pi}{k_0}\) and even about \(\tfrac{\pi}{k_0}\) by periodicity, we have \(v_{\infty}(x)+1 \geq \eps\) for all \(x \in \bigl[0,\tfrac{2\pi}{k_0}\bigr]\setminus B_r(x_\mathrm{min})\) by the maximality of \(r\). Hence, by \eqref{eq:eps-bound-2}, we obtain
    \begin{equation*}
        v(s_n,x) + 1 \geq v_{\infty}(x) +1 - |v(s_n,x) - v_{\infty}(x)|  \geq \frac{\eps}{2}
    \end{equation*}
    for all \(x \in \bigl[0,\tfrac{2\pi}{k_0}\bigr]\setminus B_r(x_\mathrm{min})\) and \(n\geq N_0\). This concludes the proof of step 2.

    \noindent \textbf{Step 3: uniform bound in \(\mathbf{W^{2,p}_{\mathrm{per}}}\).}\linkdest{uniform-bound-step3}{} The uniform bounds in \(W^{2,p}_\mathrm{per}\) follow by the same argument using that there exists a constant \(C_p>0\) such that
    \begin{equation*}
        \int_{B_r(x_{\mathrm{min}})}  |\log(|x-x_\mathrm{min}|^2)|^p \de x \leq C_p
    \end{equation*}
    for all \(1\leq p < \infty\). For the blow up of the \(W^{2,\infty}_{\mathrm{per}}\)-norm, observe that \(\log(1+v(s))\) becomes unbounded as \(\inf_{x\in \bigl(-\tfrac{\pi}{k_0},\tfrac{\pi}{k_0}\bigr)} v(s)\) approaches \(-1\).
\end{proof}

\subsection{Limit points are weak stationary solutions exhibiting film rupture}
With the uniform bounds obtained in the previous part of this section, we are now in the position to identify and study the limit points of the global bifurcation branch. In particular, we conclude from Proposition \ref{prop:uniform-bound-H2} that there is \(v_{\infty}\in \Xcal\) such that
\begin{equation*}
    v(s_n) \rightharpoonup v_{\infty} \quad \text{weakly in } H^2_{\mathrm{per}}
\end{equation*}
along a sequence \(s_n\to \infty\). We want to prove that \(h_{\infty}=1+v_{\infty}\) is a weak stationary solution to the thin-film equation \eqref{eq:thin-film-equation}. In order to pass to  the limit in the nonlinear terms of the weak formulation for \(h(s)=1+v(s)\), c.f.~Definition \ref{def:weak-solution}, we first need to provide uniform bounds for the third derivative locally on the positivity set
\begin{equation*}
    \{h_{\infty} > 0\} := \{x\in \R\, :\, h_{\infty}(x) >0\}.
\end{equation*}

\begin{lemma}\label{lem:uniform-bound-third-derivative}
    Let \(Z\subset\{h_{\infty}>0\}\) be a compact set. Then there exists a constant \(C=C(Z) >0\) such that
    \begin{equation*}
        \|\partial_x^3 v(s_n)\|_{L^2(Z)} = \|\partial_x^3 h(s_n)\|_{L^2(Z)} \leq C
    \end{equation*}
    uniformly for all \(n\in \N\).
\end{lemma}

\begin{proof}
    By Proposition \ref{prop:global-bifurcation}, \(v(s)\) is smooth. Differentiating \eqref{eq:shifted-steady-state-equation} we obtain that $\partial_x^3 v$ satisfies
    \begin{equation*}
        \partial_x^3 v(s) = g \partial_x v(s) - M \frac{\partial_x v(s)}{(1+v(s))(2+v(s))^2}.
    \end{equation*}
    Since \(Z\subset\{h_{\infty}>0\}\) is compact and \(h_{\infty}\) is continuous, there exists \(N>0\) such that \(Z\subset \{h_{\infty}> \tfrac{1}{N}\}\). Furthermore, since \(v(s_n)\) converges to \(v_{\infty}\) uniformly, we conclude that
    \begin{equation*}
        1 + v(s_n) \geq \frac{1}{2N} \quad \text{for all }x\in\{h_{\infty}> \tfrac{1}{N}\} \text{ and } n \text{ large enough}. 
    \end{equation*}
    Thus,
    \begin{equation*}
        |\partial_x^3 v(s_n)| \leq |\partial_x v(s_n)| (g + 2N M) \quad \text{for all } x\in Z.
    \end{equation*}
    Using that \(v(s_n)\) is uniformly bounded in \(W^{1,\infty}(\R)\) by Proposition \ref{prop:uniform-bound-H2}, this implies the desired uniform \(L^2\)-bound for \(\partial_x^3h(s_n)\) on \(Z\).
\end{proof}

Now, we prove the main theorem of this section on the existence of weak stationary periodic film-rupture solutions that are the limit points of the global bifurcation branches.

\begin{theorem}[Film rupture]\label{thm:weak-solutions}
    For all $k_0 > 0$ exist \(M_{\infty} > 0\) and a function \(v_{\infty} \in W^{2,p}_{\mathrm{per}}\) for all \(p\in [1,\infty)\) with \(\partial_x^3 v_{\infty} \in L^2_{\loc}(\{h_{\infty} >0\})\), where \(h_{\infty}=1+v_{\infty}\). The function \(v_{\infty}\) belongs to \(\Xcal \cap \Kcal\). In particular, $v_\infty$ is even, periodic, and non-decreasing in $[-\tfrac{\pi}{k_0},0)$. Furthermore, \(v_{\infty}\geq -1\) with \(v_{\infty}(x) = -1\) precisely in \(x= \pm\tfrac{\pi}{k_0}+2j\tfrac{\pi}{k_0}\), \(j\in \Z\).
    Eventually, we have the following convergence results (along a subsequence):
    \begin{enumerate}
        \item \(v(s) \rightharpoonup v_{\infty}\) weakly in \(W^{2,p}_{\mathrm{per}}\), in particular \(v(s) \longrightarrow v_{\infty}\) strongly in \(C^1(\R)\);
        \item \(\partial_x^3 v(s) \rightharpoonup \partial_x^3 v_{\infty}\) weakly in \(L^2_{\loc}(\{h_{\infty} >0\})\);
        \item \(h(s)^3\partial_x^3 h(s) \longrightarrow h_{\infty}^3 \partial_x^3 h_{\infty} = g\partial_x h_{\infty} - M_{\infty}\frac{h_{\infty}^2}{(1+h_{\infty})^2}\partial_x h_{\infty}\) strongly in \(C^0(\R)\), where \(h(s)=1+v(s)\) and we extend \(h_{\infty}^3\partial_x^3h_{\infty}\) continuously by zero, where \(h_{\infty}=0\).
    \end{enumerate}
    In particular, \(h_{\infty}\) is a weak stationary solution to the equation
    \begin{equation*}
        \partial_x\Bigl(h_{\infty}^3 \bigl(\partial_x^3 h_{\infty}-g\partial_x h_{\infty}\bigr) + M_{\infty}\frac{h_{\infty}^2}{(1+h_{\infty})^2}\partial_x h_{\infty}\Bigr) = 0
    \end{equation*}
    in the sense of Definition \ref{def:weak-solution}.
\end{theorem}

\begin{proof}
    By Proposition \ref{prop:uniform-bound-H2}, \((v(s),M(s))\) is uniformly bounded in \(W^{2,p}_{\mathrm{per}}\times \R\) for all \(p\in [1,\infty)\). Hence, there exists  \(v_{\infty}\in W^{2,p}_{\mathrm{per}}\), for all \(p\in [1,\infty)\), \(M_{\infty} \geq M_l > 0\), and a subsequence (not relabelled) such that
    \begin{equation*}
        (v(s),M(s)) \rightharpoonup (v_{\infty},M_{\infty}) \quad \text{weakly in } W^{2,p}_{\mathrm{per}}\times \R
    \end{equation*}
    for all \(p\in [1,\infty)\). In particular, \(v(s)\) converges to \(v_{\infty}\) strongly in \(C^1(\R)\) since the embedding \(H^2_\mathrm{per}\hookrightarrow C^1(\R)\) is compact. Furthermore, since \(\Xcal\) and \(\Kcal\) are closed with respect to weak convergence in \(H^2_{\mathrm{per}}\) and $v(s) \in \Xcal \cap \Kcal$ for all $s \geq 0$, we obtain \(v_{\infty}\in \Xcal\cap \Kcal\). By Proposition \ref{prop:film-rupture}, we obtain \(\min_{x\in \bigl[-\tfrac{\pi}{k_0},\tfrac{\pi}{k_0}\bigr]} v_{\infty} = -1\) and since \(v_{\infty}\in \Kcal\), we find that the minimum is taken at \(x=\pm \frac{\pi}{k_0}\). In view of \(|K(s)|\) being bounded uniformly by Lemma \ref{lem:uniform-bound-M}, we conclude that
    \begin{equation}\label{eq:relevant-part-of-K}
        \fint_{-\frac{\pi}{k_0}}^{\frac{\pi}{k_0}} \log(1+v(s))\de x
    \end{equation}
    is uniformly bounded. The uniform convergence of \(v(s)\) to \(v_{\infty}\) and Fatou's lemma  imply that
    \begin{equation*}
    \begin{split}
        \int_{\{v_{\infty}<0\}\cap \bigl[-\tfrac{\pi}{k_0},\tfrac{\pi}{k_0}\bigr]} |\log(1+v_{\infty})| \de x & \leq \liminf\limits_{s\to \infty} \int_{\{v_{\infty} <0\}} |\log(1+v(s))| \de x \\
        & \leq \liminf\limits_{s\to \infty} \int_{\bigl[-\tfrac{\pi}{k_0},\tfrac{\pi}{k_0}\bigr]}|\log(1+v(s))| \de x,
    \end{split}
    \end{equation*}
    which is bounded since \eqref{eq:relevant-part-of-K} is uniformly bounded and \(v(s)\) is uniformly bounded from above. This implies that \(v_{\infty}=-1\) on a set of measure zero. Since \(v_{\infty}\in \Kcal\), \(v_{\infty}=-1\) can only hold on an interval around $x = \pm \tfrac{\pi}{k_0}$, hence \(v_{\infty}=-1\) precisely on \(x=\pm \tfrac{\pi}{k_0}\).

    By Lemma \ref{lem:uniform-bound-third-derivative} (after taking another non-relabelled subsequence), we have \(\partial_x^3 v_{\infty} \in L^2_\loc(\{h_{\infty}>0\})\)  and
    \begin{equation*}
        \partial_x^3 v(s) \rightharpoonup \partial_x^3 v_{\infty} \quad \text{weakly in } L^2_\loc(\{h_{\infty}>0\}).
    \end{equation*}
    Differentiating \(\eqref{eq:shifted-steady-state-equation}\) and multiplying with \(h(s)^3 = (1+v(s))^3\), we find that
    \begin{equation*}
        h(s)^3 \partial_x^3 h(s) = gh(s)^3 \partial_x h(s) - M(s)\frac{h(s)^2}{(1+h(s))^2} \partial_xh(s).
    \end{equation*}
    Since \(v(s)\) converges in \(C^1(\R)\) (and so does \(h(s)\)), we conclude that
    \begin{equation*}
        h(s)^3\partial_x^3h(s) \longrightarrow gh_{\infty}^3 \partial_xh_{\infty} - M_{\infty}\frac{h_{\infty}^2}{(1+h_{\infty})^2}\partial_xh_{\infty} \quad \text{strongly in }C^0(\R).
    \end{equation*}
    Since \(h(s)\) converges uniformly to \(h_{\infty}\) and \(\partial_x^3 h(s)\) converges weakly in \(L^2_\loc(\{h_{\infty}>0\})\) to \(\partial_x^3 h_{\infty}\), we further find that
    \begin{equation*}
        h(s)^3 \partial_x^3 h(s) \rightharpoonup h_{\infty}^3 \partial_x^3 h_{\infty} \quad \text{weakly in } L^2_{\loc}(\{h_{\infty}>0\}).
    \end{equation*}
    In particular, 
    \begin{equation*}
        h_{\infty}^3 \partial_x^3 h_{\infty} = gh_{\infty}^3 \partial_xh_{\infty} - M_{\infty}\frac{h_{\infty}^2}{(1+h_{\infty})^2}\partial_xh_{\infty} \quad \text{almost everywhere}.
    \end{equation*}
    We may identify \(h_{\infty}^3 \partial_x^3 h_{\infty}\) with \(gh_{\infty}^3 \partial_xh_{\infty} - M_{\infty}\frac{h_{\infty}^2}{(1+h_{\infty})^2}\partial_xh_{\infty} \) (by changing it potentially on a set of measure zero). Since \(h(s)=1+v(s)\) is a smooth solution to
    \begin{equation*}
        \partial_x\Bigl(h(s)^3 \bigl(\partial_x^3 h(s)-g\partial_x h(s)\bigr) + M(s)\frac{h(s)^2}{(1+h(s))^2}\partial_x h(s)\Bigr) = 0,
    \end{equation*}
    integrating by parts yields
    \begin{equation*}
        \int_{\R} \Bigl(h(s)^3\bigl(\partial_x^3h(s) - g\partial_xh(s)\bigr) + M(s) \frac{h(s)^2}{(1+h(s))^2}\Bigr)\partial_x\phi \de x = 0
    \end{equation*}
    for every \(\phi\in H^1(\R)\) with compact support. Passing to the limit in this integral, we find that \(h_{\infty}\) solves
    \begin{equation*}
        \int_{\R} \Bigl(h_{\infty}^3\bigl(\partial_x^3h_{\infty} - g\partial_xh_{\infty}\bigr) + M_{\infty} \frac{h_{\infty}^2}{(1+h_{\infty})^2}\Bigr)\partial_x\phi \de x = 0
    \end{equation*}
    for every \(\phi\in H^1(\R)\) with compact support. This agrees with Definition \ref{def:weak-solution} since \(\{h_{\infty}=0\}\) is a null set. Hence, the theorem is proved. 
\end{proof}

\section{Stability and instability}\label{sec:stability}

In this section we discuss the (spectral) stability properties of the flat surface and of the bifurcating periodic solutions close to the bifurcation point. Following \cite[Def. 4.1.7]{kapitula2013}, we call a solution $h$ of \eqref{eq:thin-film-equation} \emph{spectrally stable} if the spectrum \(\sigma(\Lcal)\) of the  linearisation \(\Lcal\) about $h$ is contained in the closed left complex half plane, that is \(\sigma(\Lcal)\subset \{\lambda \in \C \,:\, \Re(\lambda)\leq 0\}\). Otherwise, we call \(h\) \emph{spectrally unstable}.

To study the stability of the constant steady state $\bar{h} = 1$, we linearise the thin-film equation \eqref{eq:thin-film-equation} about $\bar{h}$ and obtain the linearised equation
\begin{equation*}
    \partial_t u = \Lcal_M u, \quad \Lcal_M u = -\partial_x \Bigl(\bar{h}^3 (\partial_x^3 u -g \partial_x u) + M \frac{\bar{h}^2}{(1+\bar{h})^2}\partial_x u\Bigr) = -\partial_x^4 u  + \Bigl(g-\frac{1}{4}M\Bigr) \partial_x^2 u
\end{equation*}
The $L^2(\R)$-spectrum of $\Lcal_M$ can thus be obtained by Fourier transform and is given by
\begin{equation}\label{eq:spectrum-LM}
    \sigma_{L^2(\R)}(\Lcal_M) = \Bigl\{-\ell^4+\Bigl(\frac{1}{4}M-g\Bigr)\ell^2\,:\, \ell \in \R\Bigr\}.
\end{equation}

\begin{figure}[H]
    \centering
    \includegraphics[width=0.5\textwidth]{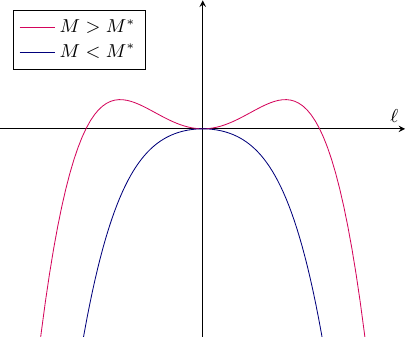}
    \caption{Spectral curves for $\mathcal{L}_M$.}
    \label{fig:spectral-curves}
\end{figure}

For $M \leq M^* = 4g$ we find that
\begin{equation*}
    \sigma_{L^2(\R)}(\Lcal_M) \subset \C_- \cup \{0\},
\end{equation*}
where $\C_- = \{\lambda \in \C \,:\, \Re(\lambda) < 0\}$ denotes the left complex half plane. This implies that $\bar{h} = 1$ is spectrally stable for $M \leq M^*$.
For $M > M^*$ there exists a $\ell_0$ such that
\begin{equation*}
    -\ell^4 + \Bigl(\frac{1}{4}M-g\Bigr)\ell^2 > 0
\end{equation*}
for all $\ell \in (-\ell_0,\ell_0) \setminus \{0\}$, see Figure \ref{fig:spectral-curves}. We point out that this type of destabilisation is also referred to as a \emph{(conserved) long-wave instability}, see \cite{cross1993} (where it is referred to as type \(\mathrm{II}_s\) instability) and \cite{schneider2017}.

In this situation, the dynamics of \eqref{eq:thin-film-equation} for $M > M^*$ close to $M^*$ is formally determined by a Cahn--Hilliard-type equation, which can be found by a weakly nonlinear analysis using a multiple-scaling expansion. Indeed, letting $\varepsilon^2 = M - M^*$ and inserting the ansatz
\begin{equation*}\label{eq:ansatz}
    h(t,x) = 1 + \eps^2 V(T,X), \quad T = \eps^4 t, X = \eps x
\end{equation*}
into \eqref{eq:thin-film-equation}, we find by equating the different powers of $\eps$ to zero that $V$ formally satisfies
\begin{equation}\label{eq:amplitude-equation}
    \partial_T V = - \partial_X^4 V - \partial_X^2 V - 2g \partial_X(V\partial_X V)
\end{equation}
to lowest order in $\eps$, see Appendix \ref{sec:appendix} and also \cite{schneider1999,doelman2009,shklyaev2017}.

We point out that the amplitude scaling in the ansatz is consistent with the amplitude scaling found on the local bifurcation curve, see Corollary \ref{cor:expansion}. Indeed, to guarantee that $\eps^2 = M - M^*$ we need to assume that $k_0 = \Ocal(\eps)$, which also yields that $\delta^2 = M^*(k_0)-M(s)$ is of order $\eps^2$. Inserting this into the expansion \eqref{eq:expansion-bifurcation} we find that the bifurcating periodic solutions have amplitude of order $\eps^2$.

Equation \eqref{eq:amplitude-equation} is known as \emph{Sivashinsky equation}, see \cite{sivashinsky1983}. It exhibits a family of stationary periodic solutions and we refer to \cite[Sec. 3.2.2.6]{shklyaev2017} for details. Additionally, the trivial steady state $V = 0$ of \eqref{eq:amplitude-equation} is unstable, but due to the negative sign of the nonlinearity, the dynamic does not saturate at a nontrivial steady state and the equation has finite-time blow up in the following sense. Fix $k > 0$. Then there exists initial data $V_0$, which are $2\pi/k$-periodic and arbitrary small in $H^3_\mathrm{per}$, such that the corresponding solution to the Sivashinsky equation \eqref{eq:amplitude-equation} blows up in finite time. To be precise, there exists a $T^\ast > 0$ such that the solution satisfies $\int_0^t \| \partial_{X}^2 V \|_{L^\infty}^2 \de t \rightarrow \infty$ as $t \rightarrow T^\ast$.
For details we refer to \cite[Theorem 3]{bernoff1995} with $\lambda = 1$ and $\sigma = - 1$. We point out that this seems consistent with the stationary results in Theorem \ref{thm:main-film-rupture}, which have an unbounded second derivative.

The physical interpretation of this finite time blow-up is the formation of film rupture in finite time, which can be observed in experiments, see for example \cite{krishnamoorthy1995,vanhook1995,vanhook1997}.

\subsection{Spectral instability of periodic solutions}

We study the spectral instability of the periodic solutions to \eqref{eq:thin-film-equation}, which were constructed in Theorem \ref{thm:local-bifurcation}. Let $k_0>0$ be a fixed wave number. Then a family of periodic periodic solutions $h_\mathrm{per}=1+v$ with period $\tfrac{2\pi}{k_0}$ bifurcates subcritically from the trivial, flat solution $\bar{h} = 1$ at $M^*(k_0) = M^\ast + 4 k_0^2$. In particular, using \eqref{eq:expansions} it satisfies the following expansion
\begin{equation*}
    h_\mathrm{per}(x) = 1 + s\cos(k_0 x) + \tau(s)
\end{equation*}
for $s > 0$ and it solves the thin-film equation \eqref{eq:thin-film-equation} with Marangoni number $M(s) = M^*(k_0) + \Ocal(s^2)$. By Remark \ref{rem:smallness-regularity} we find that $\|\tau(s)\|_{C^4} = \Ocal(s^2)$.

For $s > 0$ sufficiently small, we linearise \eqref{eq:thin-film-equation} about $h_\mathrm{per}$ and obtain the linearisation
\begin{equation*}
    \Lcal_{\mathrm{per},M} u = -\partial_x \Bigl(h_\mathrm{per}^3 (\partial_x^3 u -g \partial_x u) + 3h_\mathrm{per}^2 (\partial_x^3 h_\mathrm{per} -g \partial_x h_\mathrm{per}) u + M \frac{h_\mathrm{per}^2}{(1+h_\mathrm{per})^2}\partial_x u + M \frac{2h_\mathrm{per}\partial_x h_\mathrm{per}}{(1+h_\mathrm{per})^3} u\Bigr).
\end{equation*}
The expansion of $h_\mathrm{per}$ in $s$ yields that
\begin{equation*}
    \Lcal_{\mathrm{per},M} u = \Lcal_M u + s \Rcal u, \quad \Rcal u = \alpha_4(x,s) \partial_x^4 u + \alpha_3(x,s) \partial_x^3 u + \alpha_2(x,s) \partial_x^2 u + \alpha_1(x,s) \partial_x u + \alpha_0(x,s) u.
\end{equation*}
Here, the coefficients $\alpha_j$, $j = 0,\ldots,4$ are uniformly bounded in $L^\infty(\R)$ for $s > 0$ sufficiently small, where we use that $h_\mathrm{per} = 1 + s \cos(k_0 x) + \tau(s)$ and $\tau = \Ocal(s^2)$ in $C^4(\R)$. Since $\Rcal$ is a fourth-order operator with bounded coefficients, it is $\Lcal_{\mathrm{per},M}$-bounded in the sense that there are constants $a, b > 0$ such that 
\begin{equation}\label{eq:Lper-boundedness}
    \|\Rcal u \|_{L^2(\R)} \leq a \|u\|_{L^2(\R)} + b \|\Lcal_{\mathrm{per},M} u\|_{L^2(\R)}.
\end{equation}
Indeed, we have that the highest-order term of $\Lcal_{\mathrm{per},M}$ is given by $h_\mathrm{per}^3 \partial_x^4 u$ and $h_\mathrm{per}$ is uniformly positive for $s > 0$ sufficiently small. Therefore, the inequality \eqref{eq:Lper-boundedness} follows by standard Sobolev interpolation estimates.
In particular, since the lower bound on $h_\mathrm{per}^3$ and the upper bounds on $\alpha_j$, $j = 0,\ldots,4$, are uniform in $s$ for $s > 0$ sufficiently small, the constants $a,b$ in \eqref{eq:Lper-boundedness} can be chosen uniformly in $s$.

We now argue that $\Lcal_{\mathrm{per},M}$ has unstable spectrum for $s > 0$ sufficiently small. First, we note that since $\Lcal_M = \Lcal_{\mathrm{per},M} - s \Rcal$, where $s\Rcal$ is $\Lcal_{\mathrm{per},M}$-bounded with constants $\tilde{a}(s) = sa$ and $\tilde{b}(s) = sb$, $\Lcal_{\mathrm{per},M}$ is closed by \cite[Thm. IV.1.1]{kato1995}. Next, we use \cite[Thm. IV.2.17]{kato1995} to find that the distance of the graphs of $\Lcal_M$ and $\Lcal_{\mathrm{per},M}$ is of order $s$ for $s > 0$ sufficiently small. Finally, \cite[Thm. IV.3.1]{kato1995} implies that there exists $s_0 = s_0(k_0) > 0$ sufficiently small such that for every \(0<s < s_0\) there exists a $\lambda_0 > 0$ which lies in the $L^2(\R)$-spectrum of $\Lcal_{\mathrm{per},M}$. In particular, \(h_\mathrm{per}\) is spectrally unstable in $L^2(\R)$.

Indeed, assume that such \(\lambda_0>0\) does not exist. Then \(\sigma_{L^2}(\Lcal_{\mathrm{per},M})\cap [0,k_0^2]=\varnothing\). But now, if \(s<s_0\) and \(s_0\) is small enough, \cite[Thm. IV.3.1]{kato1995} would imply that \([0,k_0^2]\) lies in the resolvent set of \(\Lcal_M\), which contradicts \eqref{eq:spectrum-LM} since $M(s) > M^*$ for $s < s_0$.
This proves the following theorem.

\begin{theorem}[Instability]\label{thm:instability}
   Let \(k_0>0\) and 
    \begin{equation*}
        \{(v(s),M(s)) : s \in (-\eps,\eps)\} \subset \Ucal \times \R
    \end{equation*}
    the bifurcation branch obtained in Theorem \ref{thm:local-bifurcation}. Then, there exists \(s_0>0\) such that \(h(s) = 1+v(s)\) is a $L^2$-spectrally unstable solution to \eqref{eq:thin-film-equation} for every \(|s|<s_0\).
\end{theorem}

\begin{remark}
    After this work was published, we found that instability results for a class of generalised Cahn--Hilliard equations that includes our model \eqref{eq:thin-film-equation} have been discussed in \cite{laugesen2000} and \cite{laugesen2002}. There, the authors show linear and energy instability of positive periodic stationary solutions for equations of the form
    \begin{equation*}
        \partial_t h + \partial_x\bigl(F(h)\partial_x^3 h + G(h)\partial_x h\bigr) = 0
    \end{equation*}
    under co-periodic perturbations provided that  $F>0$ and $G/F$ is strictly convex. In fact, these conditions hold in our case showing that the spatially positive periodic solutions constructed in Theorem \ref{thm:global-bifurcation} are linearly and energetically unstable along the whole bifurcation curve. This does not show the instability of the film-rupture state constructed in Theorem \ref{thm:main-film-rupture} which remains an open question.
\end{remark}

%=============================================================================
%=============================================================================

\section*{Acknowledgements}
% \noindent --

B.H.~acknowledges the support of the Swedish Research Council (grant no 2020-00440).

\section*{Data Availability Statement}

The data generated for the numerical plot of the global bifurcation branch in Figure \ref{fig:numerical-continuation-plots} was obtained using the pde2path library, which can be found on \cite{zotero-3112}. The code used to generate the corresponding data is available under \url{https://github.com/Bastian-Hilder/global-bif-thermocapillary-thin-film-equation}.

\printbibliography

\appendix

\section{Formal derivation of the amplitude equation}\label{sec:appendix}

We provide the formal derivation of the amplitude equation \eqref{eq:amplitude-equation}. Therefore, we set \(\eps^2 = M-M^*\) and insert the ansatz
\begin{equation*}
    h(t,x) = 1 + \eps^2 V(T,X), \quad T = \eps^4 t, X = \eps x
\end{equation*}
into the thin-film equation \eqref{eq:thin-film-equation} and use that
\begin{equation*}
    \partial_t h = \eps^6 \partial_T V, \quad \partial_x^jh = \eps^{2+j} \partial_X V.
\end{equation*}
Rewriting \eqref{eq:thin-film-equation} via its linearisation about the constant steady state
\begin{equation*}
    \partial_t h = \Lcal_M h + \Ncal(h),
\end{equation*}
with linearisation and nonlinearity given by
\begin{equation*}
    \begin{split}
        \Lcal_M h = - \partial_x^4 h + \bigl(g-\frac{1}{4}M\bigr)\partial_x^2 u, \quad \Ncal(h) = -\partial_x\bigl((h^3-1)(\partial_x^3h-g\partial_xh) + M \bigl(\frac{h^2}{(1+h)^2}-1\bigr)\partial_x h \bigr),
    \end{split}
\end{equation*}
we obtain
\begin{equation*}
    \Lcal_M h = - \eps^6 \partial_X^4 V - \eps^6 \partial_X^2 V,
\end{equation*}
where we have used that \(M= M^* + \eps^2 = 4g + \eps^2\). Furthermore, using the Taylor expansion, it is
\begin{equation*}
    \Ncal(h) = \Ncal(1+v) = - \partial_x \bigl( (3v+3v^2+v^3)(\partial_x^3v-g\partial_xv) + M\bigl(\tfrac{v}{4} - \tfrac{v^2}{16} + \tfrac{v^4}{64} + \Ocal(|v|^5) \bigr)\partial_xv\bigr).
\end{equation*}
Inserting the ansatz, we find
\begin{equation*}
\begin{split}
    \Ncal(h) & = -\eps \partial_X\bigl( (3\eps^2V + \Ocal(\eps^4))(\eps^5\partial_X^3V - g \eps^3\partial_XV) + \tfrac{4g + \eps^2}{4} (\eps^2V +\Ocal(\eps^4))\eps^3\partial_XV\bigr) \\
    & = \eps^6 2g\partial_X(V\partial_XV) + \Ocal(\eps^8).
\end{split}
\end{equation*}
Hence, 
\begin{equation*}
    \eps^6\partial_T V = -\eps^6\bigl(\partial_X^4 V + \partial_X^2 V + 2g\partial_X(V\partial_XV)\bigr) + \Ocal(\eps^8).
\end{equation*}
Dividing by \(\eps^6\) and sending \(\eps\) to $0$, we obtain \eqref{eq:amplitude-equation}.

\vspace{2.5cm}

\end{document}